\documentclass[reqno]{amsart}
\usepackage{amsmath}
\usepackage{amssymb}
\usepackage{amsrefs}
\usepackage{mathrsfs}
\usepackage{graphicx}
\usepackage{url}
\usepackage{enumitem}
\usepackage{etoolbox}

\DeclareMathAlphabet{\mathbbold}{U}{bbold}{m}{n}

\newtheorem{theorem}{Theorem}[section]
\newtheorem{lemma}[theorem]{Lemma}
\newtheorem{cor}[theorem]{Corollary}

 \numberwithin{equation}{section}
 \numberwithin{table}{section}
 \numberwithin{figure}{section}

\newcommand{\ceil}[1]{\lceil#1\rceil}
\newcommand{\floor}[1]{\lfloor#1\rfloor}
\newcommand\thinsp{5pt}
\newcommand\TS{\rule{0pt}{2.6ex}}
\newcommand\BS{\rule[-1.2ex]{0pt}{0pt}}
\newcommand{\gfextn}[0]{pdf}
\newcommand{\figdir}[0]{.}
\newcommand\figwidth{3.5in}

\newcommand{\leg}[2]{\left(\frac{#1}{#2}\right)}
\newcommand{\abs}[1]{\left\lvert#1\right\rvert}
\newcommand{\dbars}[1]{\left\lVert#1\right\rVert}

\newcommand{\round}[1]{\left\lfloor#1\right\rceil}
\newcommand{\unimodular}[0]{\mathfrak{U}}
\newcommand{\littlewood}[0]{\mathfrak{L}}

\allowdisplaybreaks

\title{Mahler's problem and Turyn polynomials}

\author{Michael J. Mossinghoff}
\address{Center for Communications Research, Princeton, NJ, USA}
\email{m.mossinghoff@ccr-princeton.org}

\date{\today}

\begin{document}

\begin{abstract}
Mahler's problem asks for the largest possible value of the Mahler measure, normalized by the $L_2$ norm, of a polynomial with $\pm1$ coefficients and large degree.
We establish a new record value of approximately $0.951$ in this problem by analyzing certain Turyn polynomials, which are defined by cyclically shifting the coefficients of a Fekete polynomial by a prescribed amount.
Recently, Klurman, Lamzouri, and Munsch established that the distribution of values over the unit circle of Fekete polynomials of large degree is effectively modeled by a particular random point process.
We extend this analysis to the Turyn polynomials, and determine expressions for the asymptotic normalized Mahler measure of these polynomials, as well as for their normalized $L_q$ norms.
We also describe a number of calculations on the corresponding random processes, which indicate that the Turyn polynomials where the shift is approximately $1/4$ of the length have Mahler measure exceeding $95\%$ of their $L_2$ norm.
Further, we show that these asymptotic values are not disturbed by a small change to make polynomials having entirely $\pm1$ coefficients, which establishes the result on Mahler's problem.
We also estimate that the limiting value of the normalized $L_1$ norm of these polynomials exceeds $0.977$, in connection with a question of Newman.
\end{abstract}

\maketitle

\section{Introduction}\label{secIntroduction}

For a polynomial $f(x)=\sum_{k=0}^n a_k x^k = a_n \prod_{k=1}^n (x-\beta_k)\in\mathbb{C}[x]$ and a positive real number $q$, define $\dbars{f}_q$ by
\[
\dbars{f}_q = \left(\int_0^1 \abs{f(e(u))}^q du\right)^{1/q},
\]
where $e(u) := e^{2\pi i u}$.
We refer to this as the \textit{$L_q$ norm} of $f$ over the unit circle, although of course this is a norm only in the range $q\geq1$.
The limiting value as $q\to\infty$ is the supremum norm of $f$ over the unit circle,
\[
\dbars{f}_\infty=\sup_{\abs{x}=1}\abs{f(x)},
\]
and the limiting value of $\dbars{f}_q$ as $q\to0^+$ is known as the \textit{Mahler measure} of $f$, denoted $M(f)$.
The Mahler measure may be defined more simply as the geometric mean of $\abs{f(x)}$ over the unit circle,
\[
\log M(f) = \int_0^1 \log\abs{f(e(u))} \, du,
\]
and it follows easily from Jensen's formula that
\[
M(f) =  \abs{a_n} \prod_{k=1}^n \max\{1, \abs{\beta_k}\}.
\]

In 1963, Mahler \cite{Mahler63} studied the problem of determining polynomials with especially large measure, relative to their coefficient sizes.
Let $H(f)=\max_k \{\abs{a_k}\}$ denote the \textit{height} of the polynomial $f$.
Mahler proved that for each degree $n$ the supremum of $M(f)/H(f)$ over $f\in\mathbb{C}[x]$ with fixed degree is attained by a unimodular polynomial, that is, a polynomial where each coefficient $a_k$ has modulus~$1$. 
Let $\unimodular_n$ denote the set of unimodular polynomials of degree $n-1$.
Since $M(f)\leq \dbars{f}_q$ for $q\geq0$, and $\dbars{f}_2=\sqrt{n}$ for $f\in\unimodular_n$ by Parseval's formula, it is convenient to normalize its Mahler measure using the $L_2$ norm and investigate large values of $M(f)/\dbars{f}_2$.

In that same paper, Mahler cited an unpublished result of Haselgrove concerning large values of this quantity:
\[
\liminf_{n} \; \sup \left\{\frac{M(f)}{\sqrt{n}} : f\in\unimodular_n\right\}
\geq e^{-\gamma/4} = 0.86562\ldots.
\]
In 1970, Fielding \cite{Fielding70} improved this by showing that
\[
\lim_{n\to\infty} \sup \left\{\frac{M(f)}{\sqrt{n}} : f\in\unimodular_n\right\} = 1,
\]
and later Beller and Newman \cite{BN73} proved that the expression inside the limit is $1-O(\frac{\log n}{\sqrt{n}})$.

The more restrictive problem for polynomials with integer coefficients is also of interest: what is the largest possible value for the normalized Mahler measure of a polynomial in $\mathbb{Z}[x]$ and positive degree?
In view of Mahler's work, this question is of particular interest for the \textit{Littlewood polynomials}, which are polynomials $f(x)=\sum_{k=0}^{n-1} a_k x^k$ where $a_k=\pm1$ for each $k$.
We use $\littlewood$ to denote the set of these polynomials, and $\littlewood_n$ to denote the set of such polynomials of degree $n-1$.
In Peter Borwein's book \cite[Problem P10]{Borwein02}, the question of whether there exists $\epsilon>0$ such that $M(f)/\dbars{f}_2\leq1-\epsilon$ for every $f\in\littlewood$ of positive degree is known as \textit{Mahler's problem}.

Mahler's problem forms a natural companion to a famous question of D.~H. Lehmer from 1933 \cite{Lehmer33} concerning polynomials with small Mahler measure.
Specifically, Lehmer asked if there exists $\epsilon>0$ such that if $f(x)\in\mathbb{Z}[x]$ and $M(f)>1$ then $M(f)\geq1+\epsilon$.
Lehmer found that
\[
M\left(x^{10}+x^9-x^7-x^6-x^5-x^4-x^3+x+1\right) = 1.17628\ldots,
\]
and this remains the smallest known value of the Mahler measure greater than $1$ for integer polynomials.
(Lower bounds on the Mahler measure of Littlewood polynomials are studied in \citelist{\cite{BDM07}\cite{DM05}\cite{GIMPW10}}.)

The largest known value of the normalized Mahler measure of a Littlewood polynomial with positive degree occurs for
\begin{equation*}
f_B(x) = x^{12}+x^{11}+x^{10}+x^9+x^8-x^7-x^6+x^5+x^4-x^3+x^2-x+1,
\end{equation*}
where
\[
\frac{M(f_B)}{\dbars{f_B}_2} = 0.98636\ldots.
\]
Other large values can be found in \cite[Table~2]{BM08}, which lists the maximal normalized Mahler measure of a Littlewood polynomial for each degree $d\leq 25$.

Some results are known regarding Mahler's question for Littlewood polynomials with large degree.
Choi and Erdelyi \cite{CE15} proved that for each $n$ there exists $f_n\in\littlewood_n$ satisfying $M(f_n)/\sqrt{n} > \frac{1}{2} + o(1)$.
In \cite{CE14} they established a stronger result, determining the limiting mean value for the normalized Mahler measure for Littlewood polynomials of fixed degree:
\begin{equation}\label{eqnMeanMMLW}
\lim_{n\to\infty} \frac{1}{2^n} \sum_{f\in\littlewood_n} \; \frac{M(f)}{\sqrt{n}} = e^{-\gamma/2} = 0.74930\ldots,
\end{equation}
where $\gamma=0.57721\ldots$ denotes the Euler--Mascheroni constant.
(The same result in fact holds for the broader case of the unimodular polynomials \cite{CM11}.)
In 2020, Erd\'elyi \cite{Erdelyi20} showed that the Rudin--Shapiro polynomials $P_k(x)$ and $Q_k(x)$ achieve
\begin{equation*}
\lim_{k\to\infty} \frac{M(P_k)}{\dbars{P_k}_2} = \frac{M(Q_k)}{\dbars{Q_k}_2} = \sqrt{\frac{2}{e}} = 0.85776\ldots.
\end{equation*}
This established a record value in Mahler's problem for polynomials with large degree.
In this article, we show that a particular family of polynomials related to the Fekete polynomials achieves a significantly better value.

Given a prime $p$, the \textit{Fekete polynomial} for $p$ is defined using the Legendre symbol mod $p$,
\[
F_p(x) = \sum_{j=0}^{p-1} \leg{j}{p} x^j.
\]
These polynomials arise in a natural way in the study of $L$-functions associated with Dirichlet characters.
Their Mahler measure and $L_q$ norms have been the subject of significant interest and research.
Certainly $M(F_p) \leq \dbars{F_p}_2 = \sqrt{p-1}$ (and likewise $\dbars{F_p}_q \leq \sqrt{p-1}$ for $0<q\leq 2$).
Moreover, a result of Littlewood from 1961 \cite{Littlewood61} implies that there is an absolute positive constant $c_0$ so that
\begin{equation}\label{eqnLwBound}
\frac{M(F_p)}{\sqrt{p-1}} \leq 1 - c_0
\end{equation}
for all odd primes $p$.
This follows from the fact that $F_p(x)$ has coefficients $\pm1$ and that $x^p F_p(1/x) = a(p)^2 F_p(x)$, where
\begin{equation}\label{eqnap}
a(p) = \begin{cases}
1, & \textrm{if $p\equiv1\bmod 4$},\\
i, & \textrm{if $p\equiv3\bmod 4$},
\end{cases}
\end{equation}
so that $e(-pu) a(p) F_p(e(2u))$ is a real trigonometric polynomial with coefficients in $\{-2,2\}$.
(A result analogous to \eqref{eqnLwBound} holds as well for $\dbars{F_p}_q$ when $0<q<2$, with a positive constant $c_q$ that depends on $q$.)
In the other direction, in 2007 Erd\'elyi and Lubinsky \cite{EL07} established a lower bound on the Mahler measure of the Fekete polynomials, proving that for any $\epsilon>0$ one has
\[
\frac{M(F_p)}{\sqrt{p}} \geq \frac{1}{2} - \epsilon
\]
when $p$ is sufficiently large.
(Here and subsequently, it is convenient to normalize using $\sqrt{p}$, rather than the $L_2$ norm $\sqrt{p-1}$, as this does not disturb the asymptotics.)
This was improved by Erd\'elyi in 2018 \cite{Erdelyi18}, who showed that if $p$ is sufficiently large, then
\[
\frac{M(F_p)}{\sqrt{p}} \geq c_1
\]
for an absolute constant $c_1>1/2$.
Recently, Klurman, Lamzouri, and Munsch \cite{KLM23} determined much more precise information on the normalized Mahler measure of the Fekete polynomials.
By employing a particular random process, they found that
\begin{equation}\label{eqnKLM}
\lim_{p\to\infty} \frac{M(F_p)}{\sqrt{p}} \approx 0.74083\ldots\,.
\end{equation}
Note that this is not the same as the expected value of the normalized Mahler measure of a Littlewood polynomial of large degree \eqref{eqnMeanMMLW}.

In this article, we apply the method of Klurman, Lamzouri, and Munsch to a more general family of polynomials, the \textit{Turyn polynomials}, which are defined by cyclically shifting the coefficients of a Fekete polynomial by some prescribed amount.
Given a prime $p$ and an integer $t$, the Turyn polynomial $F_{p,t}(x)$ is defined by
\[
F_{p,t}(x) = \sum_{j=0}^{p-1} \leg{j + t}{p} x^j
\]
so that $F_{p,0}(x) = F_p(x)$.
We find that the limiting normalized Mahler measure of $F_{p,t}$ increases over $0\leq t < p/4$, so the value at $t\approx p/4$ is of particular interest.
Letting $\round{x}$ denote the integer nearest $x$ (and using $\floor{x}$ as its value at a half-integer), using our main result from Section~\ref{secRandProc} and the calculations in Section~\ref{secCalculations}, we find that
\begin{equation}\label{eqnTurynMeas}
\lim_{p\to\infty} \frac{M(F_{p,\round{p/4}})}{\sqrt{p}} \approx 0.951\ldots\,.
\end{equation}

The Turyn polynomials are not Littlewood polynomials, owing to a single coefficient of $0$, so these fall just short of fulfilling the requirement in Mahler's problem.
In general, changing a single coefficient of an integer polynomial by $1$ can have an enormous effect on its Mahler measure: $M((x+1)^n) = M(x+1)^n = 1$, but $M((x+1)^n-1) \approx 1.4^n$, see \cite{Pritsker08}.
However, for the Turyn polynomials we show that changing this single coefficient of $0$ to $\pm1$ to form a companion Littlewood polynomial does not affect the normalized measure in the limit (see Section~\ref{secRandProc}).
The Littlewood polynomials obtained in this way when the Turyn shift amount is $t=\round{p/4}$ thus establish a new record value in Mahler's problem for the normalized measure of an infinite family of Littlewood polynomials.

The $L_q$ norms of the Fekete and Turyn polynomials have also seen substantial interest, especially the case $q=4$ and $t = \round{p/4}$.
Determining Littlewood polynomials with small normalized $L_4$ norm is equivalent to the \textit{merit factor problem} of Golay, who stated this problem in an equivalent way involving sequences over $\{-1,1\}$ having small aperiodic autocorrelations.
In 1983, Golay \cite{Golay83} attributed to Turyn the empirical observation that very good merit factors are exhibited when shifting a Legendre sequence by approximately $1/4$ of its length.
Using a heuristic assumption, Golay computed in effect that the asymptotic value of the normalized $L_4$ norm of the Turyn polynomials is
\[
\lim_{p\to\infty} \frac{\dbars{F_{p,t}}_4}{\sqrt{p}} = \left(8\left(\frac{t}{p}-\frac{1}{4}\right)^2+\frac{7}{6}\right)^{1/4}
\]
for $0\leq t< p/2$, and of course this achieves its minimal value $(7/6)^{1/4}$ when $t/p=1/4$.
This was proved unconditionally in 1988 by H{\o}holdt and Jensen \cite{HJ88}, and in 2002 Borwein and Choi \cite{BC02} refined this by determining a precise formula for $\dbars{F_{p,t}}_4$, with particular attention to the cases $t=0$ and $t=\round{p/4}$.
For more information on the merit factor, autocorrelations of binary sequences, and their relations to problems in analysis, number theory, and engineering, we refer the reader to \citelist{\cite{Borwein02}\cite{BM08}\cite{Jedwab05}\cite{Katz18}\cite{Schmidt16}}.

In 2017, G\"unther and Schmidt \cite{GS17} greatly generalized these results by determining the limiting value of the normalized $L_{2k}$ norm of the Fekete and Turyn polynomials when $k$ is a positive integer.
They proved that there is a function $\phi_k : \mathbb{R}\to\mathbb{R}$ so that $(L_{2k}(F_{p,\round{\alpha p}})/\sqrt{p})^{2k} = \phi_k(\alpha)$, described an efficient method to calculate $\phi_k(0)$, and showed that $\phi_k(\alpha)$ attains its minimum at $\alpha=1/4$ for $k\in\{2,3,4\}$.
They also conjectured that this remains the case for every integer $k\geq5$.

The study of the merit factor and $L_{2k}$ norms of the Turyn polynomials, and the fact that particularly small normalized norms are achieved by these polynomials when $\alpha=1/4$, motivated the study of their Mahler measure here: one might expect that a polynomial with a small normalized $L_{2k}$ norm with $k\geq2$ would also exhibit a large normalized Mahler measure.
We find that this is indeed the case by extending the method of Klurman, Lamzouri, and Munsch to the Turyn polynomials, and we determine an expression for the limiting value of the normalized measure in terms of the relative shift $t/p$.

In \cite{KLM23} the authors also determined an expression for the limiting value of $\dbars{F_p}_q/\sqrt{p}$, for all $q>0$.
We generalize this result to the case of Turyn polynomials $F_{p,t}$ as well.
This extends to their companion Littlewood polynomials too, as with the Mahler measure.

The case of $L_q$ norms with $q$ an odd integer also arises in the literature, especially the case $q=1$.
In 1965, Newman \cite{Newman65} proved that there exists a constant $c_1>0$ so that for each $n\geq2$ there exists a unimodular polynomial $f\in\mathfrak{U}_n$ with $\dbars{f}_1>\sqrt{n}-c_1$, but much less is known for the Littlewood case.
In \cite{Newman60}, Newman cited a conjecture (without attribution) that there exists a positive constant $\epsilon_1$ so that $\dbars{f}_1/\dbars{f}_2 < 1-\epsilon_1$ for every $f\in\mathfrak{L}$ of positive degree.
We call this $L_1$ analogue of Mahler's question \textit{Newman's problem} (see also \cite[Section~5]{BM08}).
We obtain an estimate for an upper bound on a putative $\epsilon_1$ in this problem by using the Turyn polynomials, finding that
\begin{equation}\label{eqnTurynL1}
\lim_{p\to\infty} \frac{\dbars{F_{p,\round{p/4}}}_1}{\sqrt{p}} \approx 0.9775\ldots,
\end{equation}
and that this value is not disturbed by altering the $F_{p,\round{p/4}}$ slightly to make Littlewood polynomials, just as with the Mahler measure.

The remainder of this article is organized in the following way.
Section~\ref{secProperties} establishes some analytic properties of the Turyn polynomials, including a bound on their supremum norm, which generalizes a result of Montgomery. 
After these preparations, our main results are stated in Section~\ref{secRandProc} regarding the Mahler measure and $L_q$ norms of the Turyn polynomials, which generalize results for the Fekete polynomials.
We establish an expression for the limiting value of $M(F_{p,\round{\alpha p}})/\sqrt{p}$ as $p\to\infty$ for any fixed value of $\alpha$ in $[0,1]$, and similarly for $\dbars{F_{p,\round{\alpha p}}}_q/\sqrt{p}$ for any $q>0$: see Theorem~\ref{thmTuryn}.
This statement is proved in Section~\ref{secProofThm} using the approach of \cite{KLM23}.
Further, we show that these asymptotic values for the Turyn polynomials are undisturbed by a small alteration to create companion Littlewood polynomials, in connection the problems of Mahler and Newman: see Corollary~\ref{corTuryn} in Section~\ref{secRandProc}.
This statement is proved in Section~\ref{secProofCor}.
Section~\ref{secCalculations} describes a number of computations for estimating these limiting values.
In particular, we obtain the estimate \eqref{eqnTurynMeas} for the Turyn polynomials when $t=\round{p/4}$, and for the Fekete polynomials we find the limiting value is approximately $0.740$, which is close to the value \eqref{eqnKLM} from \cite{KLM23}.
The same section also summarizes some calculations with the $L_1$ and $L_3$ norms, and obtains \eqref{eqnTurynL1}.
Finally, Section~\ref{secFuture} briefly indicates some potential directions for future work regarding a promising family of generalized Turyn polynomials.

\section{Properties of Turyn polynomials}\label{secProperties}

We establish some analytic properties of Turyn polynomials.
Let $\zeta_p$ denote the $p$th root of unity $e(1/p)$.
Recall that Gauss established that if $p$ is an odd prime then $F_p(\zeta_p) = a(p) \sqrt{p}$, with $a(p)$ as in \eqref{eqnap}, and it is easily shown from this that $F_p(\zeta_p^k) = \leg{k}{p} F_p(\zeta_p)$.
We first record a generalization of these facts for Turyn polynomials.

\begin{lemma}\label{lemmaGaussSum}
Suppose $p$ is an odd prime and $t$ is an integer.
Then
\begin{equation}\label{eqnTurynz}
F_{p,t}(\zeta_p) = a(p) \zeta_p^{-t} \sqrt{p}.
\end{equation}
In addition, for each integer $k\geq0$ we have
\begin{equation}\label{eqnTurynzk}
F_{p,t}(\zeta_p^k) = \leg{k}{p}\zeta_p^{-(k-1)t} F_{p,t}(\zeta_p).
\end{equation}
\end{lemma}

\begin{proof}
The identity \eqref{eqnTurynz} follows easily from Gauss' formula.
For \eqref{eqnTurynzk}, this is trivial when $p\mid k$, since $F_{p,t}(1)=0$.
If $p\nmid k$, we calculate
\begin{align*}
F_{p,t}(\zeta_p^k)
   &= \sum_{j=0}^{p-1}\leg{j+t}{p} \zeta_p^{jk}\\
   &= \leg{k}{p} \sum_{j=0}^{p-1} \leg{jk+tk}{p} \zeta_p^{jk}\\
   &= \leg{k}{p} \zeta_p^{-(k-1)t} \sum_{j=0}^{p-1} \leg{jk+t(k-1)+t}{p}\zeta_p^{jk+t(k-1)}\\
   &= \leg{k}{p}\zeta_p^{-(k-1)t} F_{p,t}(\zeta_p).
\qedhere
\end{align*}
\end{proof}

Next, from Gauss' formula one observes that $\abs{F_p(\zeta_p^k)} = \sqrt{p}$ for $1\leq k < p$.
With so many values of the same modulus exhibited by $F_p(e(u))$ at equally spaced points $e(u)$ around the unit circle, it is natural to ask if $\dbars{F_p}_\infty = O(\sqrt{p})$, but in 1980 Montgomery \cite{Montgomery80} showed that this is not the case, proving that if $p$ is sufficiently large then
\[
\frac{2}{\pi} \sqrt{p} \log\log p < \dbars{F_p}_\infty < C \sqrt{p} \log p
\]
for an absolute constant $C$.
From Lemma~\ref{lemmaGaussSum}, we see that the Turyn polynomials also have the property that $\abs{F_{p,t}(\zeta_p^k)} = \sqrt{p}$ for $1\leq k<p$.
We require an upper bound for the supremum norm of the Turyn polynomials over the unit circle and establish it here by adapting Montgomery's argument.

\begin{theorem}\label{thmMontgomery}
The Turyn polynomials satisfy $\dbars{F_{p,t}}_\infty = O\left(\sqrt{p} \log p\right)$, where the implied constant is independent of $t$.
\end{theorem}

\begin{proof}
From \eqref{eqnTurynzk}, we have
\[
\leg{k}{p} = \zeta_p^{(k-1)t}\frac{F_{p,t}(\zeta_p^k)}{F_{p,t}(\zeta_p)},
\]
so
\begin{align*}
F_{p,t}(e(u))
&= e(-tu) \sum_{j=t}^{p+t-1} \leg{j}{p} e(ju)\\
&= \frac{e(-tu)}{F_{p,t}(\zeta_p)} \sum_{j=t}^{p+t-1} \zeta_p^{(j-1)t} F_{p,t}(\zeta_p^j) e(ju)\\
&= \frac{e(-tu)}{F_{p,t}(\zeta_p)} \sum_{j=t}^{p+t-1} \zeta_p^{(j-1)t} e(ju) \sum_{k=t}^{p+t-1} \leg{k}{p} \zeta_p^{j(k-t)} \\
&= \frac{e(-tu)}{\zeta_p^t F_{p,t}(\zeta_p)} \sum_{k=t}^{p+t-1} \leg{k}{p} \sum_{j=t}^{p+t-1}  e\left(j(u+k/p)\right)\\
&= \frac{1-e(pu)}{\zeta_p^t F_{p,t}(\zeta_p)}  \sum_{k=t}^{p+t-1} \leg{k}{p} \frac{\zeta_p^{kt}}{1-e(u+k/p)}\\
&= \frac{1-e(pu)}{2i\zeta_p^t F_{p,t}(\zeta_p)}\sum_{k=t}^{p+t-1} \leg{k}{p}\zeta_p^{kt}\left(i-\cot(\pi(u+k/p)) \right)\\
&= \frac{1-e(pu)}{2}\left(\leg{t}{p} - \frac{1}{i a(p)\sqrt{p}}\sum_{k=t}^{p+t-1}\leg{k}{p}\zeta_p^{kt}\cot(\pi(u+k/p))\right),
\end{align*}
where the last line follows from Lemma~\ref{lemmaGaussSum}.
Let $\dbars{x}$ denote the distance from $x$ to the integer nearest $x$, so $\dbars{x}=\min\{x-\floor{x},\ceil{x}-x\}$.
Then
\begin{align*}
\abs{F_{p,t}(e(u))}
&\leq \abs{\sin(\pi \dbars{p u})} \left(1 + \frac{1}{\sqrt{p}}\sum_{k=t}^{p+t-1} \abs{\cot(\pi(u+k/p))}\right)\\
&\leq \abs{\sin(\pi \dbars{p u})} \left(1 + \frac{1}{\pi\sqrt{p}} \sum_{k=t}^{p+t-1} \dbars{u + k/p}^{-1}\right)\\
&\ll \abs{\sin(\pi \dbars{p u})}\left(1 + \sqrt{p}\dbars{pu}^{-1} + \sqrt{p}\log p\right).
\end{align*}
The statement follows.
\end{proof}

Last, we require a formula for the value of a Turyn polynomial near a $p$th root of unity.
In \cite{KLM23}, for integer $k$ and real number $x\in[0,1]$, the authors defined the auxiliary function
\[
G_p(k,x) = \frac{F_p(\zeta_p^{k+x})}{F_p(\zeta_p)},
\]
where we use $\zeta_p^{k+x}$ to denote $e((k+x)/p)$.
They employed an interpolation from \cite{CGPS00} to determine an expression for this function.
We require an analogue of that interpolation for the Turyn polynomials.
For an integer $k$ and real number $x\in[0,1]$, let
\begin{equation}\label{eqnGpt}
G_{p,t}(k,x) := \zeta_p^{(k-1)t} \frac{F_{p,t}(\zeta_p^{k+x})}{F_{p,t}(\zeta_p)}.
\end{equation}

\begin{theorem}\label{thmAbsTuryn}
Let $k$ be an integer and let $x\in(0,1)$.
Then
\[
G_{p,t}(k,x) =
\frac{e(x)-1}{p}\sum_{\abs{j}<p/2} \leg{k-j}{p} \frac{\zeta_p^{jt}}{\zeta_p^{j+x}-1}.
\]
\end{theorem}

\begin{proof}
Using Lagrange interpolation, we have
\[
F_{p,t}(z) = \sum_{\ell=0}^{p-1} F_{p,t}(\zeta_p^\ell)
\prod_{\substack{j=0\\j\neq\ell}}^{p-1} \frac{z-\zeta_p^j}{\zeta_p^\ell-\zeta_p^j}.
\]
Since
\[
\prod_{\substack{j=0\\j\neq\ell}}^{p-1} (z-\zeta_p^j) = \frac{z^p-1}{z-\zeta_p^\ell}
\;\;\; \textrm{and} \;\;\;
\prod_{\substack{j=0\\j\neq\ell}}^{p-1} (\zeta_p^\ell-\zeta_p^j) =
\lim_{z\to\zeta_p^\ell} \frac{z^p-1}{z-\zeta_p^\ell} = \frac{p}{\zeta_p^{\ell}},
\]
using \eqref{eqnTurynzk} we find that
\[
\frac{F_{p,t}(z)}{F_{p,t}(\zeta_p)}
= \sum_{\ell=0}^{p-1} \leg{\ell}{p}\zeta_p^{-(\ell-1)t} \frac{\zeta_p^\ell(z^p-1)}{p(z-\zeta_p^\ell)}
= \frac{\zeta_p^t}{p}(z^p-1) \sum_{\ell=0}^{p-1} \leg{\ell}{p} \frac{\zeta_p^{\ell(1-t)}}{z-\zeta_p^\ell}.
\]
Now fix an integer $k$ and write $z=\zeta_p^k e(x/p)=\zeta_p^{k+x}$ with $x\in(0,1)$.
Then
\begin{align*}
\frac{F_{p,t}(\zeta_p^{k+x})}{F_{p,t}(\zeta_p)} &=
\frac{\zeta_p^t (e(x)-1)}{p}
\sum_{\ell=0}^{p-1} \leg{\ell}{p} \frac{\zeta_p^{-\ell t}}{\zeta_p^{k-\ell+x}-1}\\
&=
\frac{\zeta_p^t (e(x)-1)}{p} \sum_{\abs{k-\ell}<p/2} \leg{\ell}{p} \frac{\zeta_p^{-\ell t}}{\zeta_p^{k-\ell+x}-1},
\end{align*}
and substituting $j=k-\ell$ we obtain the result.
\end{proof}

The problem of computing the Mahler measure of Turyn polynomials thus reduces to a computation regarding the $G_{p,t}$ functions: using \eqref{eqnTurynz} and \eqref{eqnGpt}, we see that
\begin{equation}\label{eqnTurynToInterp}
\begin{split}
\log M(F_{p,t}) &= \int_0^1 \log\abs{F_{p,t}(e(u))}\,du\\
&= \frac{1}{2}\log p + \frac{1}{p} \sum_{k=0}^{p-1} \int_0^1 \log\abs{G_{p,t}(k, x)}\,dx.
\end{split}
\end{equation}
A similar reduction applies for computing $\dbars{F_{p,t}}_q$.
An insight of \cite{KLM23} is that the interpolating formula from \cite{CGPS00} for $G_p(k,x)$ in these calculations can be effectively modeled by using a particular random process.
This procedure generalizes to the Turyn case, as we summarize in the next section.

\section{Mahler measure and $L_q$ norms}\label{secRandProc}

We describe our main results.

\subsection{Turyn polynomials}\label{subsecTuryn}
Since
\begin{equation}\label{eqnApprox1}
\frac{1}{p(\zeta_p^{j+x}-1)} = \frac{1}{2\pi i(j+x)} + O\left(\frac{1}{p}\right),
\end{equation}
and for $\alpha\in[0,1]$, with $\delta>0$ fixed and $j=O(p^{1-\delta})$ one has
\begin{equation}\label{eqnApprox2}
\zeta_p^{\round{\alpha p}j} = e(\alpha j) + O\left(\frac{1}{p^{\delta}}\right),
\end{equation}
Theorem~\ref{thmAbsTuryn} suggests that studying the sum
\[
\sum_{j\in\mathbb{Z}} \delta_j \frac{e(\alpha j)(e(x)-1)}{2\pi i (j+x)},
\]
with each $\delta_j\in\{-1,1\}$, could provide information on $G_{p,\round{\alpha p}}(k,x)$, and thus on the Mahler measure of the Turyn polynomials via \eqref{eqnTurynToInterp}, as well as their $L_q$ norms.
To this end, for each $\alpha\in[0,1]$ we introduce the random point process
\begin{equation}\label{eqnRandProc}
G_{\mathbb{X},\alpha}(x) := \sum_{j\in\mathbb{Z}} \frac{e(\alpha j)(e(x)-1)}{2\pi i (j+x)} \mathbb{X}(j), \;\;\; x\in[0,1],
\end{equation}
where the $\{\mathbb{X}(j)\}_{j\in\mathbb{Z}}$ are independent, identically distributed random variables taking the values $\pm1$ with equal probability $1/2$.

We refer to $\alpha$ as the \textit{relative shift} of the Turyn polynomial.
In \cite{KLM23}, the authors studied the case $\alpha=0$, and established that the sequence of processes $(G_{p,0}(k,x))$ converges in law to the process $G_{\mathbb{X},0}(x)$ in the space $C[0,1]$ of continuous functions on $[0,1]$.
We find that this method may be adapted to $(G_{p,\alpha}(k,x))$ and $G_{\mathbb{X},\alpha}(x)$ for any $\alpha\in[0,1]$, and establish the following statements in Section~\ref{secProofThm}.

\begin{theorem}\label{thmTuryn}
Let $\alpha\in[0,1]$.
\begin{enumerate}[label=(\roman*)]
\item\label{partPhi}
If $\varphi : C[0,1] \to \mathbb{C}$ is bounded and continuous and $x\in[0,1]$, then
\[
\lim_{p\to\infty} \frac{1}{p} \sum_{k=0}^{p-1} \varphi\left(G_{p,\round{\alpha p}}(k, x)\right) = \mathbb{E}(\varphi(G_{\mathbb{X},\alpha}(x))).
\]
\item\label{partLq}
For any $q>0$, the $L_q$ norm of the Turyn polynomial with relative shift $\alpha$ satisfies
\[
\lim_{p\to\infty} \frac{\dbars{F_{p,\round{\alpha p}}}_q}{\sqrt{p}}
= \kappa_q(\alpha)
:= \left(\int_0^1 \mathbb{E}(\abs{G_{\mathbb{X},\alpha}(x)}^q)\, dx\right)^{1/q}.
\]
\item\label{partMeasure}
The Mahler measure of the Turyn polynomial with relative shift $\alpha$ satisfies
\[
\lim_{p\to\infty} \frac{M(F_{p,\round{\alpha p}})}{\sqrt{p}}
= \kappa_0(\alpha)
:= \exp\left(\int_0^1 \mathbb{E}(\log \abs{G_{\mathbb{X},\alpha}(x)})\, dx\right).
\]
\end{enumerate}
\end{theorem}

\subsection{Companion Littlewood polynomials}\label{subsecTurynLW}
The polynomials $F_{p,t}(x)$ are not Littlewood polynomials, owing to the single coefficient of $0$ at the $x^{p-t}$ term when $0<t\leq p$.
(It is convenient here to represent the Fekete polynomial using $t=p$.)
For each prime $p$ and each $t$ in this range, we define a pair of companion Littlewood polynomials by changing this coefficient to $\pm1$.
Let
\begin{equation}\label{eqnFptpm}
F_{p,t}^{\pm}(x) := F_{p,t}(x) \pm x^{p-t},
\end{equation}
and define
\begin{equation}\label{eqnGptpm}
G_{p,t}^{\pm}(k,x) := \zeta_p^{(k-1)t} \frac{F_{p,t}^{\pm}(\zeta_p^{k+x})}{F_{p,t}(\zeta_p)}.
\end{equation}
We establish the following corollary in Section~\ref{secProofCor} in connection with the problems of Mahler and Newman.

\begin{cor}\label{corTuryn}
Let $\alpha\in[0,1]$.
\begin{enumerate}[label=(\roman*)]
\item\label{partPhiLW}
If $\varphi : C[0,1] \to \mathbb{C}$ is bounded and continuous and $x\in[0,1]$, then
\[
\lim_{p\to\infty} \frac{1}{p} \sum_{k=0}^{p-1} \varphi\left(G_{p,\round{\alpha p}}^{\pm}(k, x)\right) = \mathbb{E}(\varphi(G_{\mathbb{X},\alpha}(x))).
\]
\item\label{partLqLW}
For any $q>0$, the $L_q$ norm of a Littlewood polynomial corresponding to the Turyn polynomial with relative shift $\alpha$ satisfies
\[
\lim_{p\to\infty} \frac{\big\lVert F_{p,\round{\alpha p}}^{\pm} \big\rVert_q}{\sqrt{p}}
= \kappa_q(\alpha).
\]
\item\label{partMeasureLW}
The Mahler measure of a Littlewood polynomial corresponding to the Turyn polynomial with relative shift $\alpha$ satisfies
\[
\lim_{p\to\infty} \frac{M(F_{p,\round{\alpha p}}^{\pm})}{\sqrt{p}}
= \kappa_0(\alpha).
\]
\end{enumerate}
\end{cor}

\section{Proof of Theorem~\ref{thmTuryn}}\label{secProofThm}

We adapt the method of \cite{KLM23} to establish Theorem~\ref{thmTuryn}.
Before proceeding with the proofs of these statements, we require a preliminary result concerning the random process $G_{\mathbb{X},\alpha}(x)$ from \eqref{eqnRandProc}.

\begin{lemma}\label{lemGXalpha}
The functions $\Re G_{\mathbb{X},\alpha}(x)$ and $\Im G_{\mathbb{X},\alpha}(x)$ are almost surely real analytic on $0\leq x\leq1$.
\end{lemma}

\begin{proof}
We consider $\Re G_{\mathbb{X},\alpha}(x)$; the argument for $\Im G_{\mathbb{X},\alpha}(x)$ is essentially the same.
Let $r_\alpha(j, x) = \sin(2\pi(\alpha j+x)) - \sin(2\pi\alpha j)$ so that
\[
\Re G_{\mathbb{X},\alpha} =
\sum_{\abs{j}\leq1} \frac{r_\alpha(j,x)}{2\pi(j+x)} \mathbb{X}(j)
+ \sum_{\abs{j}\geq2} \frac{r_\alpha(j,x)}{2\pi(j+x)} \mathbb{X}(j).
\]
The first sum is certainly analytic over all possible realizations of $\mathbb{X}(-1)$, $\mathbb{X}(0)$, and $\mathbb{X}(1)$.
Let $R_{\mathbb{X},\alpha}(x)$ denote the second sum.
This is almost surely bounded on $[0,1]$, and for each $k\geq1$, and all realizations of the $\mathbb{X}(j)$, we have
\[
\abs{R^{(k)}_{\mathbb{X},\alpha}(x)}
\leq
k! (2\pi)^{k-1} \sum_{\abs{j}\geq 2} \frac{1}{(\abs{j}-1)^{k+1}}
=
2k! (2\pi)^{k-1} \zeta(k+1)
<
(2\pi)^k k!
\]
uniformly for $x\in[0,1]$, so $\Re G_{\mathbb{X},\alpha}(x)$ is almost surely real analytic.
\end{proof}

Each part of Theorem~\ref{thmTuryn} is considered in turn below.

\subsection{Proof of Part~\ref{partPhi}}\label{subsecPartPhi}
For fixed $\alpha\in[0,1]$, we consider $G_{p,\round{\alpha p}}(k,x)$ with $x\in[0,1]$ as a random process (in $k$) over $\mathbb{F}_p$ with the uniform distribution.
Part~\ref{partPhi} is equivalent to the statement that the sequence (in $p$) of processes $(G_{p,\round{\alpha p}}(k,x))$ converges in law to the random process $G_{\mathbb{X},\alpha}(x)$ with $x\in[0,1]$.
Convergence in law is established by first proving the convergence of the finite moments of the processes as $p\to\infty$ in Lemma~\ref{lemMoments}, and then showing that the former sequence is relatively compact in Lemma~\ref{lemEquicont}.

\begin{lemma}\label{lemMoments}
Let $L$ be a positive integer, let $0\leq x_1<x_2<\cdots<x_L\leq1$ be real numbers, and let $r_1,\ldots, r_L$ and $s_1,\ldots,s_L$ be nonnegative integers.
Let $r=r_1+\cdots+r_L$ and $s=s_1+\cdots+s_L$.
Then
\begin{equation*}
\begin{split}
&\frac{1}{p} \sum_{k=0}^{p-1} \prod_{\ell=1}^L G_{p,\round{\alpha p}}(k, x_\ell)^{r_\ell} \overline{G_{p,\round{\alpha p}}(k, x_\ell)}^{s_\ell}\\
&\qquad= \mathbb{E}\left(\prod_{\ell=1}^L G_{\mathbb{X},\alpha}(x_\ell)^{r_\ell} \overline{G_{\mathbb{X},\alpha}(x_\ell)}^{s_\ell}\right) + O_{r,s}\left(\frac{(\log p)^{r+s}}{p^{1/3}}\right).
\end{split}
\end{equation*}
\end{lemma}

\begin{proof}
For convenience, let
\[
\beta_{p,\alpha}(j, x) := \frac{\zeta_p^{j\round{\alpha p}}(e(x)-1)}{p(\zeta_p^{j+x}-1)}.
\]
Using Theorem~\ref{thmAbsTuryn}, we compute
\begin{align*}
&\prod_{\ell=1}^L G_{p,\round{\alpha p}}(k, x_\ell)^{r_\ell} \overline{G_{p,\round{\alpha p}}(k, x_\ell)}^{s_\ell}\\
&\qquad =
\prod_{\ell=1}^L \left(\sum_{\abs{j}<p/2}\leg{k-j}{p} \beta_{p,\alpha}(j, x_\ell) \right)^{r_\ell}
\prod_{\ell=1}^L \left(\sum_{\abs{j}<p/2}\leg{k-j}{p} \overline{\beta_{p,\alpha}(j, x_\ell)} \right)^{s_\ell}\\
&\qquad =
\sum_{\mathbf{j}_1,\ldots,\mathbf{j}_L} \prod_{\ell=1}^L \beta_{p,\alpha}(\mathbf{j}_\ell, x_\ell) \prod_{\sigma=1}^{r_\ell+s_\ell} \leg{k-j_{\ell,\sigma}}{p},
\end{align*}
where the integer vectors $\mathbf{j}_\ell = (j_{\ell,1} , \ldots, j_{\ell,r_\ell+s_\ell})$ range over all possibilities where each component satisfies $\abs{j_{\ell,\sigma}}<p/2$, and
\[
\beta_{p,\alpha}(\mathbf{j}_\ell, x_\ell) := \prod_{\sigma=1}^{r_\ell} \beta_{p,\alpha}(j_{\ell,\sigma}, x_\ell) \prod_{\sigma=1}^{s_\ell} \overline{\beta_{p,\alpha}(j_{\ell,r_\ell+\sigma}, x_\ell)}.
\]
Thus
\begin{equation*}
\begin{split}
&\frac{1}{p}\sum_{k=0}^{p-1} \prod_{\ell=1}^L G_{p,\round{\alpha p}}(k, x_\ell)^{r_\ell} \overline{G_{p,\round{\alpha p}}(k, x_\ell)}^{s_\ell}\\
&\qquad =
\sum_{\mathbf{j}_1,\ldots,\mathbf{j}_L} \left(\prod_{\ell=1}^L \beta_{p,\alpha}(\mathbf{j}_\ell, x_\ell)\right) \frac{1}{p} \sum_{k=0}^{p-1} \prod_{\ell=1}^L \prod_{\sigma=1}^{r_\ell+s_\ell} \leg{k-j_{\ell,\sigma}}{p}.
\end{split}
\end{equation*}
Employing Weil's bound for character sums \cite[Theorem 2C]{Schmidt76}, we have
\[
\frac{1}{p} \sum_{k=0}^{p-1} \prod_{\ell=1}^L \prod_{\sigma=1}^{r_\ell+s_\ell} \leg{k-j_{\ell,\sigma}}{p}
=
\mathbb{E}\left(\prod_{\ell=1}^L \prod_{\sigma=1}^{r_j+s_j} \mathbb{X}(j_{\ell,\sigma}) \right) + O\left(\frac{r+s}{\sqrt{p}}\right).
\]
Now
\[
\frac{1}{\sqrt{p}} \sum_{\mathbf{j}_1,\ldots,\mathbf{j}_L} \prod_{\ell=1}^L \abs{\beta_{p,\alpha}(\mathbf{j}_\ell, x_\ell)}
=
\frac{1}{\sqrt{p}} \prod_{\ell=1}^L \left(\sum_{\abs{j}<p/2} \abs{\beta_{p,\alpha}(j,x_\ell)} \right)^{r_\ell+s_\ell},
\]
and $\abs{\beta_{p,\alpha}(j,x)}\ll1/\abs{j}$ for $\abs{j}\geq2$, independent of $\alpha$, so
 $\sum_{\abs{j}<p/2} \abs{\beta_{p,\alpha}(j,x)} \ll \log p$, and therefore
\begin{equation}\label{eqnMoments}
\begin{split}
&\frac{1}{p}\sum_{k=0}^{p-1} \prod_{\ell=1}^L G_{p,\round{\alpha p}}(k, x_\ell)^{r_\ell} \overline{G_{p,\round{\alpha p}}(k, x_\ell)}^{s_\ell}\\
&\qquad =
\mathbb{E}\left(\sum_{\mathbf{j}_1,\ldots,\mathbf{j}_L} \left(\prod_{\ell=1}^L \beta_{p,\alpha}(\mathbf{j}_\ell, x_\ell)\right) \prod_{\ell=1}^L \prod_{\sigma=1}^{r_j+s_j} \mathbb{X}(j_{\ell,\sigma}) \right) + O_{r,s}\left(\frac{(\log p)^{r+s}}{\sqrt{p}}\right)\\
&\qquad =
\mathbb{E}\left(\prod_{\ell=1}^L G^*_{\mathbb{X},\alpha,p}(x_\ell)^{r_\ell} \overline{G^*_{\mathbb{X},\alpha,p}(x_\ell)}^{s_\ell}\right)  + O_{r,s}\left(\frac{(\log p)^{r+s}}{\sqrt{p}}\right),
\end{split}
\end{equation}
where
\[
G^*_{\mathbb{X},\alpha,p}(x) := \sum_{\abs{j}<p/2} \beta_{p,\alpha}(j,x) \mathbb{X}(j).
\]
We also define $G_{\mathbb{X},\alpha,p}(x)$ for $x\in[0,1]$ by truncating the series \eqref{eqnRandProc} defining $G_{\mathbb{X},\alpha}(x)$,
\[
G_{\mathbb{X},\alpha,p}(x) := \sum_{\abs{j}<p^{2/3}} \frac{e(\alpha j)(e(x)-1)}{2\pi i (j+x)} \mathbb{X}(j).
\]
Using Minkowski's inequality, the fact that the $\mathbb{X}(j)$ are independent, \eqref{eqnApprox1}, and \eqref{eqnApprox2} with $\delta=1/3$, we have
\begin{align*}
&\mathbb{E}\left(\abs{G_{\mathbb{X},\alpha}(x) - G^*_{\mathbb{X},\alpha,p}(x)}^2\right)\\
&\qquad \leq
\mathbb{E}\left(\abs{G_{\mathbb{X},\alpha}(x) - G_{\mathbb{X},\alpha,p}(x)}^2\right)
+
\mathbb{E}\left(\abs{G_{\mathbb{X},\alpha,p}(x) - G^*_{\mathbb{X},\alpha,p}(x)}^2\right)\\
&\qquad \leq
\sum_{\abs{j}>p^{2/3}} \frac{\abs{e(x)-1}^2}{4\pi^2(j+x)^2}
+
\sum_{\abs{j}<p^{2/3}} \abs{e(x)-1}^2 \abs{\frac{e(j\round{\alpha p}/p)}{p(e((j+x)/p)-1)} - \frac{e(\alpha j)}{2\pi i(j+x)}}^2\\
&\qquad\qquad +
\sum_{p^{2/3}<\abs{j}<p/2} \abs{\frac{e(x)-1}{p(e((j+x)/p)-1)}}^2\\
&\qquad \ll
\sum_{\abs{j}>p^{2/3}} \frac{1}{j^2}
+ \sum_{\abs{j}<p^{2/3}} \frac{\sin^2(\pi x)}{p^{2/3}(j+x)^2}
+ \sum_{p^{2/3}<\abs{j}<p/2} \frac{\sin^2(\pi x)}{(j+x)^2}
\ll \frac{1}{p^{2/3}}
\end{align*}
uniformly in $x$.
Let $\mathbb{Y}_{\ell,\alpha} := G_{\mathbb{X},\alpha}(x_\ell)-G^*_{\mathbb{X},\alpha,p}(x_\ell)$, so that $\mathbb{E}(\abs{\mathbb{Y}_{\ell,\alpha}}^2) \ll 1/p^{2/3}$ for $1\leq\ell\leq L$.
Noting from Cauchy--Schwarz that $\mathbb{E}(\abs{\mathbb{Y}_{\ell_1,\alpha}\cdots\mathbb{Y}_{\ell_k,\alpha}}) \ll_k 1/p^{1/3}$ for any $k\geq1$ and $1\leq \ell_1 \leq \cdots \leq \ell_k \leq L$, we compute
\begin{align*}
&\mathbb{E}\left(\prod_{\ell=1}^L G_{\mathbb{X},\alpha}(x_\ell)^{r_\ell} \overline{G_{\mathbb{X},\alpha}(x_\ell)}^{s_\ell}\right)
=
\mathbb{E}\left(\prod_{\ell=1}^L (G^*_{\mathbb{X},\alpha,p}(x_\ell)+\mathbb{Y}_{\ell,\alpha})^{r_\ell} (\overline{G^*_{\mathbb{X},\alpha,p}(x_\ell)+\mathbb{Y}_{\ell,\alpha}})^{s_\ell}\right)\\
&\quad =
\sum_{\substack{(k_1,\ldots,k_{2L})\\0\leq k_\ell\leq r_\ell\\0\leq k_{L+\ell}\leq s_\ell}}
\hspace{-2ex}\mathbb{E}\left(\prod_{\ell=1}^L \binom{r_\ell}{k_\ell} \binom{s_\ell}{k_{L+\ell}} \mathbb{Y}_{\ell,\alpha}^{k_\ell} \overline{\mathbb{Y}}_{\ell,\alpha}^{k_{L+\ell}} G^*_{\mathbb{X},\alpha,p}(x_\ell)^{r_\ell-k_\ell} \overline{G^*_{\mathbb{X},\alpha,p}(x_\ell)}^{s_\ell-k_{L+\ell}}\right)\\
&\quad =
\mathbb{E}\left(\prod_{\ell=1}^L G^*_{\mathbb{X},\alpha,p}(x_\ell)^{r_\ell} \overline{G^*_{\mathbb{X},\alpha,p}(x_\ell)}^{s_\ell}\right)
 + O_{r,s}\Biggl((\log p)^{r+s}
  \hspace{-4ex}\sum_{\substack{(k_1,\ldots,k_{2L})\neq\mathbf{0}\\0\leq k_\ell\leq r_\ell\\0\leq k_{L+\ell}\leq s_\ell}}
  \hspace{-2.5ex}\mathbb{E}\left(\prod_{\ell=1}^L \abs{\mathbb{Y}_{\ell,\alpha}}^{k_\ell+k_{L+\ell}}\right)
\Biggr)\\
&\quad =
\mathbb{E}\left(\prod_{\ell=1}^L G^*_{\mathbb{X},\alpha,p}(x_\ell)^{r_\ell} \overline{G^*_{\mathbb{X},\alpha,p}(x_\ell)}^{s_\ell}\right) + O_{r,s}\left(\frac{(\log p)^{r+s}}{p^{1/3}}\right),
\end{align*}
where we used that $\abs{G^*_{\mathbb{X},\alpha,p}(x)} \ll \log p$ for $x\in[0,1]$.
The statement follows by combining this with \eqref{eqnMoments}.
\end{proof}

Proving that a sequence of random processes is relative compact is equivalent (by Prokhorov's theorem) to establishing its tightness, and a criterion of Kolmogorov provides a method for this \cite[Chapter~1]{Krylov02}.
This criterion asserts that if there exist constants $\delta>0$, $\gamma>0$, and $C\geq0$ such that
\[
\mathbb{E}\left(\abs{G_{p,\round{\alpha p}}(k, x) - G_{p,\round{\alpha p}}(k, y)}^\gamma\right) \leq C(y-x)^{1+\delta} 
\]
for all $p\geq p_0$ and any $x<y$ in $[0,1]$, then $(G_{p,\round{\alpha p}}(k, x))$ is tight.
The following lemma establishes this equicontinuity result with $p_0=3$, $\gamma=2$, and $\delta=1/2$.

\begin{lemma}\label{lemEquicont}
Let $\alpha\in[0,1]$.
There exists an absolute constant $C>0$ such that if $p\geq3$ is prime, $t$ is a nonnegative integer, and $0\leq x<y\leq1$, then
\[
\frac{1}{p} \sum_{k=0}^{p-1} \abs{G_{p,\round{\alpha p}}(k,x) - G_{p,\round{\alpha p}}(k,y)}^2 \leq C(y-x)^{3/2}.
\]
\end{lemma}

\begin{proof}
By combining Bernstein's theorem for polynomials with Theorem~\ref{thmMontgomery}, we have
\[
\sup_{\abs{z}=1} \abs{F'_{p,\round{\alpha p}}(z)} \ll p^{3/2} \log p,
\]
so using the mean value theorem we conclude
\begin{align*}
\abs{G_{p,\round{\alpha p}}(k,x) - G_{p,\round{\alpha p}}(k,y)}
&= \frac{1}{\sqrt{p}} \abs{F_{p,\round{\alpha p}}(\zeta_p^{k+x}) - F_{p,\round{\alpha p}}(\zeta_p^{k+y})}\\
&\ll p\log p \abs{\frac{k+x}{p} - \frac{k+y}{p}}\\
&\ll (y-x)\log p.
\end{align*}
If $y-x \leq 1/(\log p)^4$, then
\[
\frac{1}{p} \sum_{k=0}^{p-1} \abs{G_{p,\round{\alpha p}}(k,x) - G_{p,\round{\alpha p}}(k,y)}^2 \ll (y-x)^2 (\log p)^2 \leq (y-x)^{3/2},
\]
so suppose $y-x > 1/(\log p)^4$.

Let $b_j(x) =\beta_{p,0}(j,x) = \frac{e(x)-1}{p(\zeta_p^{j+x}-1)}$ so
\begin{align*}
&\abs{G_{p,\round{\alpha p}}(k,x) - G_{p,\round{\alpha p}}(k,y)}^2 =
\Bigl|\sum_{\abs{j}<p/2} \zeta_p^{j\round{\alpha p}}\leg{k-j}{p} \left( b_j(x) - b_j(y)\right)\Bigr|^2\\
&\quad=
\sum_{\abs{j}<p/2} \abs{b_j(x) - b_j(y)}^2 +\\
&\qquad
\!\!\sum_{\substack{\abs{j}<p/2\\\abs{j'}<p/2\\j\neq j'}}  \!\! \zeta_p^{(j-j')\round{\alpha p}}\leg{(k-j)(k-j')}{p} (b_j(x) - b_j(y))(\overline{b_{j'}(x) - b_{j'}(y)}).
\end{align*}
Since $0<\abs{j-j'}<p$, from a result of Gauss we have $\sum_{k=0}^{p-1} \leg{(k-j)(k-j')}{p}=-1$ for $p\geq3$, so using Cauchy--Schwarz we find that
\begin{align*}
&\frac{1}{p} \sum_{k=0}^{p-1} \abs{G_{p,\round{\alpha p}}(k,x) - G_{p,\round{\alpha p}}(k,y)}^2\\
&\qquad =
\sum_{\abs{j}<p/2} \abs{b_j(x) - b_j(y)}^2
- \frac{1}{p} 
\sum_{\substack{\abs{j}<p/2\\\abs{j'}<p/2\\j\neq j'}} \zeta_p^{(j-j')\round{\alpha p}} (b_j(x) - b_j(y))(\overline{b_{j'}(x) - b_{j'}(y)})\\
&\qquad \leq
\sum_{\abs{j}<p/2} \abs{b_j(x) - b_j(y)}^2
+ \frac{1}{p} 
\sum_{\substack{\abs{j}<p/2\\\abs{j'}<p/2\\j\neq j'}}  \abs{b_j(x) - b_j(y)}\abs{b_{j'}(x) - b_{j'}(y)}\\
&\qquad < 2\sum_{\abs{j}<p/2} \abs{b_j(x) - b_j(y)}^2.
\end{align*}
Since this is independent of $\alpha$, the remainder of the proof is identical to the last part of the proof of \cite[Lemma 4.6]{KLM23}, so is only briefly summarized here.
Let $g_j(x) := \frac{e(x)-1}{2\pi i(j+x)}$ so that $b_j(x) = g_j(x) + O(1/p)$.
Since $\abs{g'_j(x)} \ll \min(1,1/j)$ for all $x\in[0,1]$ and $\abs{j}<p/2$, there exist absolute constants $c_1$ and $c_2$ such that $\abs{b_j(x) - b_j(y)} \leq c_1 (y-x) \min\left(1, 1/j\right) + c_2/p$, so
\[
\frac{1}{p} \sum_{k=0}^{p-1} \abs{G_{p,\round{\alpha p}}(k,x) - G_{p,\round{\alpha p}}(k,y)}^2
\ll
\!\! \!\sum_{0<j<p/2} \!\!\! \frac{(y-x)^2}{j^2} + \frac{1}{p^2} \!\! \sum_{\abs{j}<p/2} 1
\ll (y-x)^2 + \frac{1}{p}.
\]
Since $p^{-2/3} \ll (\log p)^{-4}$ and $y-x>(\log p)^{-4}$, we have $1/p \ll (y-x)^{3/2}$, so
\[
\frac{1}{p} \sum_{k=0}^{p-1} \abs{G_{p,\round{\alpha p}}(k,x) - G_{p,\round{\alpha p}}(k,y)}^2 \ll (y-x)^{3/2}.\qedhere
\]
\end{proof}
This completes the proof of part~\ref{partPhi}.

\subsection{Proof of Part~\ref{partLq}}\label{subsecPartLq}
Given $q>0$, we select $\varphi_q$ to be the bounded and continuous functional on $C[0,1]$ given by
\[
\varphi_q(f) = \int_0^1 \abs{f(u)}^q\,du.
\]
Then for fixed $\alpha\in[0,1]$, using \eqref{eqnTurynz} and \eqref{eqnGpt} we have
\begin{align*}
\int_0^1 \abs{F_{p,\round{\alpha p}}(e(u))}^q\,du
&=
\frac{1}{p} \sum_{k=0}^{p-1} \int_0^1 \abs{F_{p,\round{\alpha p}}(\zeta_p^{k+x})}^q\,dx\\
&=
p^{q/2}\cdot\frac{1}{p} \sum_{k=0}^{p-1} \int_0^1 \abs{G_{p,\round{\alpha p}}(k,x)}^q\,dx.
\end{align*}
After suitable rearrangement and taking limits,  part~\ref{partLq} then follows from part~\ref{partPhi}.

\subsection{Proof of Part~\ref{partMeasure}}\label{subsecPartMeasure}
We cannot apply part~\ref{partPhi} directly in this case, since the functional $\varphi$ on $C[0,1]$ given by $\varphi(f) = \int_0^1 \log\abs{f(u)}\,du$ is not continuous.
For convenience, let
\begin{equation}\label{eqnHdefn}
H_{p,\alpha}(k,x) := \frac{2\pi i G_{p,\round{\alpha p}}(k,x)}{e(x)-1},
\qquad
H_{\mathbb{X},\alpha}(x) := \frac{2\pi i G_{\mathbb{X},\alpha}(x)}{e(x)-1}.
\end{equation}
Given $\epsilon>0$.
Certainly $H_{p,\alpha}(k,x)$ is continuous over the interval $[\epsilon,1-\epsilon]$, and Lemma~\ref{lemGXalpha} implies that $H_{\mathbb{X},\alpha}(x)$ is almost surely continuous on this same interval.
By the arguments of Section~\ref{subsecPartPhi} therefore the sequence of processes $(H_{p,\round{\alpha p}}(k,x))$ converges in law to the process $H_{\mathbb{X},\alpha}(x)$ in the space $C[\epsilon,1-\epsilon]$.
Define the functional $\varphi_\epsilon$ on this space by
\[
\varphi_\epsilon(f) := \int_\epsilon^{1-\epsilon} \log(\abs{f(x)}) \mathbbold{1}_{\abs{f(x)}\geq\epsilon}\,dx.
\]
Since this is bounded and continuous, by part~\ref{partPhi} we have
\begin{equation}\label{eqnHmost}
\begin{split}
&\lim_{p\to\infty} \frac{1}{p} \sum_{k=0}^{p-1} \int_\epsilon^{1-\epsilon} \log\left(\abs{H_{p,\alpha}(k,x)}\right) \mathbbold{1}_{\abs{H_{p,\alpha}(x)}\geq\epsilon}\,dx\\
&\qquad =
 \int_\epsilon^{1-\epsilon} \mathbb{E}\left(\log\left(\abs{H_{\mathbb{X},\alpha}(x)} \right)\mathbbold{1}_{\abs{H_{\mathbb{X},\alpha}(x)}\geq\epsilon}\right)\,dx.
\end{split}
\end{equation}
It remains to show that the remaining integrals of interest are small: integrating over $[0,1]$ where the function in question is less than $\epsilon$ in absolute value, and integrating over $[0,\epsilon]$ and over $[1-\epsilon,1]$ where it is at least $\epsilon$ in absolute value.
These two cases are treated in turn by Lemmas~\ref{lemSmallValsInt} and~\ref{lemLargeValsAtEndsInt}.
First however we define
\[
\widetilde{H}_{p,\alpha}(k,x) := \sum_{\abs{j}<p/2} \leg{k-j}{p} \frac{\zeta_p^{j\round{\alpha p}}}{j+x}
\]
and show that it suffices to consider $\widetilde{H}_{p,\alpha}(k,x)$ in place of $H_{p,\alpha}(k,x)$ in our subsequent arguments.

\begin{lemma}\label{lemHstar}
For any $\alpha\in[0,1]$ and any integer $0\leq k<p$, we have
\begin{gather}\label{Hstar1}
\sup_{0<x<1} \abs{\widetilde{H}'_{p,\alpha}(k,x) - H'_{p,\alpha}(k,x)} = O\left(\frac{1}{p}\right),\\\label{Hstar2}
\sup_{0<x<1} \abs{\widetilde{H''}_{p,\alpha}(k,x) - H''_{p,\alpha}(k,x)} = O\left(\frac{1}{p^2}\right).
\end{gather}
\end{lemma}

\begin{proof}
Using Theorem~\ref{thmAbsTuryn} and \eqref{eqnHdefn}, we compute
\[
\widetilde{H}'_{p,\alpha}(x) = -\sum_{\abs{j}<p/2} \leg{k-j}{p} \frac{\zeta_p^{j\round{\alpha p}}}{(j+x)^2}
\]
and
\begin{align*}
H'_{p,\alpha}(k,x) &=
-\frac{\pi^2}{p^2}\sum_{\abs{j}<p/2} \leg{k-j}{p} \frac{\zeta_p^{j\round{\alpha p}}}{\sin^2(\pi(j+x)/p)}\\
&= -\sum_{\abs{j}<p/2} \leg{k-j}{p} \zeta_p^{j\round{\alpha p}} \left(\frac{1}{(j+x)^2} + O\left(\frac{1}{p^2}\right)\right)
\end{align*}
so \eqref{Hstar1} follows easily.
Differentiating again, we find
\[
\widetilde{H}''_{p,\alpha}(x) = 2\sum_{\abs{j}<p/2} \leg{k-j}{p} \frac{\zeta_p^{j\round{\alpha p}}}{(j+x)^3}
\]
and
\begin{align*}
H''_{p,\alpha}(k,x) &=
\frac{2\pi^3}{p^3}\sum_{\abs{j}<p/2} \leg{k-j}{p} \frac{\zeta_p^{j\round{\alpha p}}\cos(\pi(j+x)/p)}{\sin^3(\pi(j+x)/p)}\\
&= 2\sum_{\abs{j}<p/2} \leg{k-j}{p} \zeta_p^{j\round{\alpha p}} \left(\frac{1}{(j+x)^3} + O\left(\frac{1}{p^3}\right)\right),
\end{align*}
which yields \eqref{Hstar2}.
\end{proof}

We now consider the two remaining integrals.

\begin{lemma}\label{lemSmallValsInt}
Let $\alpha\in[0,1]$.
For primes $p\to\infty$, each $0\leq k<p$, and sufficiently small $\epsilon>0$, we have
\begin{gather}\label{eqnSmallH}
\int_0^1 \log\bigl(\big| H_{p,\alpha}(k,x)\big|\bigr) \mathbbold{1}_{\abs{H_{p,\alpha}(k,x)}\leq\epsilon}\, dx = O\left(\epsilon^{0.24}\right),\\\label{eqnsmallwtH}
\int_0^1 \log\bigl(\big|\widetilde{H}_{p,\alpha}(k,x)\big|\bigr) \mathbbold{1}_{\abs{\widetilde{H}_{p,\alpha}(k,x)}\leq\epsilon}\, dx = O\left(\epsilon^{0.24}\right).
\end{gather}
In addition, almost surely
\begin{equation}\label{eqnSmallHX}
\int_0^1 \log\bigl(\big| H_{\mathbb{X},\alpha}(x)\big|\bigr) \mathbbold{1}_{\abs{H_{\mathbb{X},\alpha}(x)}\leq\epsilon}\, dx = O\left(\epsilon^{0.24}\right).
\end{equation}
\end{lemma}

\begin{proof}
We first argue that for each $p$, $\alpha$, and $k$, at least one of $|\widetilde{H}'_{p,\alpha}(k,x)|$ and $|\widetilde{H}''_{p,\alpha}(k,x)|$ is bounded away from $0$ over $0<x<1$.
Suppose first $0<k<p-1$.
Fix a prime $p$, and for convenience let $t=\round{\alpha p}$.
Then
\begin{align*}
\abs{\widetilde{H}'_{p,\alpha}(k,x)}
&=
\abs{\sum_{\abs{j}<p/2} \leg{k-j}{p} \frac{\zeta_p^{jt}}{(j+x)^2}}\\
&\geq
\abs{\frac{1}{x^2} + \leg{k(k+1)}{p}\frac{\zeta_p^{-t}}{(x-1)^2}} - \sum_{\substack{\abs{j}<p/2\\j\neq0,-1}} \frac{1}{(j+x)^2}.
\end{align*}
Since $0\leq x\leq1$, we have
\begin{align*}
\sum_{\substack{\abs{j}<p/2\\j\neq0,-1}} \frac{1}{(j+x)^2}
&\leq
\sum_{j\geq1} \left(\frac{1}{(j+x)^2} + \frac{1}{(j+1-x)^2}\right)\\
&\leq
\sum_{j\geq1} \left(\frac{1}{j^2} + \frac{1}{(j+1)^2}\right)
= 2\zeta(2)-1.
\end{align*}
When $\leg{k}{p}=\leg{k+1}{p}$, we compute
\[
\abs{\frac{1}{x^2} + \frac{\zeta_p^{-t}}{(x-1)^2}}^2 = 
\frac{1}{x^4} + \frac{2\cos(2\pi t)}{x^2(x-1)^2} + \frac{1}{(x-1)^4},
\]
and for fixed $t$ this expression achieves its minimum over $0\leq x\leq1$ at $x=1/2$.
It follows that for each fixed $t$ we have
\[
\abs{\frac{1}{x^2} + \frac{\zeta_p^{-t}}{(x-1)^2}} \geq 8 \abs{\cos(\pi t/p)}.
\]
Similarly, when $\leg{k}{p}=-\leg{k+1}{p}$, we find
\[
\abs{\frac{1}{x^2} - \frac{\zeta_p^{-t}}{(x-1)^2}} \geq 8\abs{\sin(\pi t/p)}.
\]
Thus for $0\leq x\leq1$ we have
\[
\abs{\widetilde{H}'_{p,\alpha}(k,x)} \geq
\begin{cases}
8\abs{\cos(\pi t/p)} - 2\zeta(2) + 1, & \textrm{if $\leg{k}{p}=\leg{k+1}{p}$},\\[\thinsp]
8\abs{\sin(\pi t/p)} - 2\zeta(2) + 1, & \textrm{if $\leg{k}{p}=-\leg{k+1}{p}$}.
\end{cases}
\]
Treating $\widetilde{H}''_{p,\alpha}(k,x)$ in the same way, we note
\[
\sum_{\substack{\abs{j}<p/2\\j\neq0,-1}} \frac{1}{(j+x)^3}
= \sum_{1\leq j<p/2} \frac{1}{(j+x)^3} - \sum_{2\leq j<p/2} \frac{1}{(j-x)^3} < \zeta(3)
\]
and find
\[
\frac{1}{2}\abs{\widetilde{H}''_{p,\alpha}(k,x)} \geq
\begin{cases}
16\abs{\sin(\pi t/p)} - \zeta(3), & \textrm{if $\leg{k}{p}=\leg{k+1}{p}$},\\[\thinsp]
16\abs{\cos(\pi t/p)} - \zeta(3), & \textrm{if $\leg{k}{p}=-\leg{k+1}{p}$}.
\end{cases}
\]
It follows that for any fixed $\alpha\in[0,1]$, and any prime $p$, at least one of $|\widetilde{H}'_{p,\alpha}(k,x)|$ and $|\widetilde{H}''_{p,\alpha}(k,x)|$ exceeds $5.4645\ldots$ over the entire interval $[0,1]$.

For $k=0$, we employ
\[
\abs{\widetilde{H}'_{p,\alpha}(0,x)}
\geq
 \abs{\frac{\zeta_p^t}{(x+1)^2} + \leg{-1}{p}\frac{\zeta_p^{-t}}{(x-1)^2}} - \sum_{2\leq \abs{j}<p/2} \frac{1}{(j+x)^2}.
\]
We first compute
\[
\abs{\frac{\zeta_p^t}{(x+1)^2} + \leg{-1}{p}\frac{\zeta_p^{-t}}{(x-1)^2}}^2 = \frac{1}{(x+1)^4} + \frac{1}{(x-1)^4} + \leg{-1}{p}\frac{2\cos(4\pi t/p)}{(x^2-1)^2}.
\]
For any fixed $t$ this is increasing in $x$, and we record that its value at $x=0$ is either $4\cos^2(2\pi t/p)$ or $4\sin^2(2\pi t/p)$, depending on whether $\leg{-1}{p}$ is $1$ or $-1$ respectively, and that its value at $x=1/2$ is either
\[
\frac{64}{81}\left(16 + 9\cos^2(2\pi t/p)\right)
\quad \mathrm{or} \quad
\frac{64}{81}\left(16 + 9\sin^2(2\pi t/p)\right),
\]
depending on the same condition.
If $1/2\leq x\leq 1$, then
\begin{align*}
\sum_{2\leq \abs{j}<p/2} \frac{1}{(j+x)^2}
&=
\frac{1}{(2-x)^2} + \frac{1}{(3-x)^2} + \sum_{j\geq3} \left(\frac{1}{(j-(1-x))^2} + \frac{1}{(j+(1-x))^2}\right)\\
&\leq
\frac{5}{4} + \sum_{j\geq3} \left(\frac{1}{(j-1/2)^2} + \frac{1}{(j+1/2)^2}\right),
\end{align*}
since the first two terms on the right side are increasing on $[1/2,1]$, and each term in the sum is decreasing there.
Since $\sum_{j\geq1} 1/(j-1/2)^2 = \sum_{j\geq0} 1/(j+1/2)^2 = \pi^2/2$, we conclude
\[
\sum_{2\leq \abs{j}<p/2} \frac{1}{(j+x)^2} \leq \frac{5}{4} + \pi^2 - \left(4+\frac{4}{9}\right) - \left(4 + \frac{4}{9} + \frac{4}{25}\right) = \pi^2 - \frac{7019}{900}.
\]
We then find that
\[
\abs{\widetilde{H}'_{p,\alpha}(0,x)}
\geq
\frac{8}{9}\sqrt{16 + 9\cos^2(2\pi t/p)} - \left(\pi^2 - \frac{7019}{900}\right)
\geq \frac{10219}{900}-\pi ^2 = 1.4848\ldots
\]
when $p\equiv1\bmod4$, and that the same bound holds when $p\equiv3\bmod4$.
If $0\leq x\leq 1/2$, then
\begin{align*}
\sum_{2\leq \abs{j}<p/2} \frac{1}{(j+x)^2}
&\leq
\sum_{j\geq 2} \left(\frac{1}{(j+x)^2} +  \frac{1}{(j-x)^2}\right)\\
&\leq
\sum_{j\geq 2}\left( \frac{1}{(j+1/2)^2} +  \frac{1}{(j-1/2)^2}\right)
= \pi^2 - \frac{76}{9},
\end{align*}
so
\[
\abs{\widetilde{H}'_{p,\alpha}(0,x)}
\geq
\begin{cases}
2\abs{\cos(2\pi t/p)} - \left(\pi^2 - \frac{76}{9}\right), & \textrm{if $\leg{-1}{p}=1$},\\[\thinsp]
2\abs{\sin(2\pi t/p)} - \left(\pi^2 - \frac{76}{9}\right), & \textrm{if $\leg{-1}{p}=-1$}.
\end{cases}
\]
Since $|\widetilde{H}'_{p,\alpha}(0,x)|>0$ across this interval only for particular ranges of $\alpha$, we turn again to the second derivative.
Here we compute in a similar way that
\[
\abs{\frac{\zeta_p^t}{(x+1)^3} + \leg{-1}{p}\frac{\zeta_p^{-t}}{(x-1)^3}}^2 = \frac{1}{(x+1)^6} + \frac{1}{(x-1)^6} + \leg{-1}{p}\frac{2\cos(4\pi t/p)}{(x^2-1)^3}
\]
and
\begin{align*}
\sum_{2\leq \abs{j}<p/2} \frac{1}{(j+x)^3}
&=
\frac{1}{(x-2)^3} + \frac{1}{(x-3)^3} + \sum_{2\leq j<p/2} \frac{1}{(j+x)^3} - \sum_{4\leq j<p/2} \frac{1}{(j-x)^3}\\
&\leq
-\frac{35}{216} + \sum_{j\geq2} \frac{1}{j^3}
=
\zeta(3) - \frac{251}{216}.
\end{align*}
It follows that
\[
\frac{1}{2}\abs{\widetilde{H}''_{p,\alpha}(0,x)}
\geq
\begin{cases}
2\abs{\sin(2\pi t/p)} - \left(\zeta(3) - \frac{251}{216}\right), & \textrm{if $\leg{-1}{p}=1$},\\[\thinsp]
2\abs{\cos(2\pi t/p)} - \left(\zeta(3) - \frac{251}{216}\right), & \textrm{if $\leg{-1}{p}=-1$}
\end{cases}
\]
for $0\leq x\leq1/2$,
so at least one of $|\widetilde{H}'_{p,\alpha}(0,x)|$ and $|\widetilde{H}''_{p,\alpha}(0,x)|$ exceeds $0.5498\ldots$ over this interval.

For $k=p-1$, we use
\begin{equation}\label{eqnHpm1}
\abs{\widetilde{H}'_{p,\alpha}(p-1,x)}
\geq
 \abs{\frac{1}{x^2} +\leg{-1}{p} \frac{\zeta_p^{-2t}}{(x-2)^2}} - \sum_{\substack{\abs{j}<p/2\\j\neq0,-1,-2}} \frac{1}{(j+x)^2}.
\end{equation}
The first term above is decreasing over $(0,1]$.
Its value at $x=1/2$ is either
\[
\frac{8}{9}\sqrt{16+9\cos^2(2\pi t/p)\vphantom{\sin^2}}
\quad\textrm{or}\quad
\frac{8}{9}\sqrt{16+9\sin^2(2\pi t/p)},
\]
depending on whether $\leg{-1}{p}=\pm1$.
At $x=1$, it is $2\abs{\cos(2\pi t/p)}$ or $2\abs{\sin(2\pi t/p)}$, again depending on $\leg{-1}{p}$.
The sum in \eqref{eqnHpm1} is
\[
\sum_{1\leq j<p/2} \frac{1}{(j+x)^2} + \sum_{3\leq j<p/2} \frac{1}{(j-x)^2} \leq
\begin{cases}
\frac{2\pi^2}{3} - \frac{40}{9}, & \textrm{if $0\leq x\leq 1/2$},\\[\thinsp]
\frac{2\pi^2}{3} - 5, & \textrm{if $1/2\leq x\leq 1$}.
\end{cases}
\]
Over $0\leq x\leq1/2$, it follows that $|\widetilde{H}'_{p,\alpha}(p-1,x)| \geq 8-2\pi^2/3 = 1.4202\ldots\,$.
On the remaining range, $|\widetilde{H}'_{p,\alpha}(p-1,x)| > 0$ only for certain $\alpha$, and we turn once more to the second derivative.
Since
\[
\frac{1}{2}\abs{\widetilde{H}''_{p,\alpha}(p-1,x)}
\geq
 \abs{\frac{1}{x^3} +\leg{-1}{p} \frac{\zeta_p^{-2t}}{(x-2)^3}} - \sum_{\substack{\abs{j}<p/2\\j\neq0,-1,-2}} \frac{1}{(j+x)^3},
\]
the first term is bounded below by $2\abs{\sin(2\pi t/p)}$ or $2\abs{\cos(2\pi t/p)}$ depending on $\leg{-1}{p}$, and since over $1/2\leq x\leq1$
\[
\sum_{\substack{\abs{j}<p/2\\j\neq0,-1,-2}} \frac{1}{(j+x)^3}
= \!\!\sum_{1\leq j<p/2} \frac{1}{(j+x)^3}
- \!\!\sum_{3\leq j<p/2} \frac{1}{(j-x)^3}
< \sum_{j\geq1} \frac{1}{(j+\frac{1}{2})^3}
= 7\zeta(3) - 8,
\]
we have
\[
\frac{1}{2} \abs{\widetilde{H}''_{p,\alpha}(p-1,x)} \geq
\begin{cases}
2\abs{\sin(2\pi t/p)} - \left(7\zeta(3) - 8\right), & \textrm{if $\leg{-1}{p}=1$},\\[\thinsp]
2\abs{\cos(2\pi t/p)} - \left(7\zeta(3) - 8\right), & \textrm{if $\leg{-1}{p}=-1$}
\end{cases}
\]
for $1/2\leq x\leq1$.
It follows that at least one of $|\widetilde{H}'_{p,\alpha}(p-1,x)|$ and $|\widetilde{H}''_{p,\alpha}(p-1,x)|$ exceeds $0.3339\ldots$ over $[1/2,1]$.

We may now complete the proof of \eqref{eqnsmallwtH}.
If $|\widetilde{H}_{p,\alpha}(k,x)|>\epsilon$ for $0\leq x\leq1$ there is nothing to show.
Suppose first $0<k<p$.
Since $\lim_{x\to0^+} |\widetilde{H}_{p,\alpha}(k,x)| = \lim_{x\to1^-} |\widetilde{H}_{p,\alpha}(k,x)| = \infty$, by taking $\epsilon$ sufficiently small we may assume there exist $0<x_1<x_2<x_3<1$ such that $|\widetilde{H}_{p,\alpha}(k,x_1)|=|\widetilde{H}_{p,\alpha}(k,x_3)|=\epsilon$ and $\widetilde{H}_{p,\alpha}(k,x_2)=0$.
It suffices to show that $\int_{x_2}^{x_3} \log |\widetilde{H}_{p,\alpha}(k,x)|\,dx = O\left(\epsilon^{0.24}\right)$.
With $\epsilon$ sufficiently small, we have $|\widetilde{H}'_{p,\alpha}(k,x)| > 0$ for all $x\in[x_2,x_3]$ and we may assume that
\[
\abs{\log |\widetilde{H}_{p,\alpha}(k,x)| \, \mathbbold{1}_{|\widetilde{H}_{p,\alpha}(k,x)|\leq\epsilon}} \leq |\widetilde{H}_{p,\alpha}(k,x)|^{-1/100}.
\]
By H\"older's inequality, we have
\begin{align*}
&\abs{\int_{x_2}^{x_3} \log |\widetilde{H}_{p,\alpha}(k,x)| \, \mathbbold{1}_{|\widetilde{H}_{p,\alpha}(k,x)|\leq\epsilon}\,dx}
\leq
\int_{x_2}^{x_3} |\widetilde{H}'_{p,\alpha}(k,x)|^{-1/4} \cdot \frac{|\widetilde{H}'_{p,\alpha}(k,x)|^{1/4}}{|\widetilde{H}_{p,\alpha}(k,x)|^{1/100}} \, dx\\
&\qquad \leq
\left(\int_{x_2}^{x_3} |\widetilde{H}'_{p,\alpha}(k,x)|^{-1/3} \, dx\right)^{3/4}
\left(\int_{x_2}^{x_3} \frac{|\widetilde{H}'_{p,\alpha}(k,x)|}{|\widetilde{H}_{p,\alpha}(k,x)|^{1/25}} \, dx\right)^{1/4}.
\end{align*}
The second term is $O(\epsilon^{0.24})$ by direct evaluation.
If $|\widetilde{H}'_{p,\alpha}(k,x)|\gg1$ over $[0,1]$ then the first term is $O(1)$.
Otherwise, if $\widetilde{H}'_{p,\alpha}(k,x)$ vanishes at some point $x_0$, then because
\[
\widetilde{H}'_{p,\alpha}(k,x) \approx (x-x_0) \widetilde{H}''_{p,\alpha}(k,x)
\]
for $x$ near $x_0$, and because $|\widetilde{H}''_{p,\alpha}(k,x)|\gg1$ over $[0,1]$ by the prior work, we have
\[
\int_{x_0}^x  |\widetilde{H}'_{p,\alpha}(k,u)|^{-1/3}\,du \ll \int_{x_0}^x \abs{u-x_0}^{-1/3}\,du \ll \abs{x-x_0}^{2/3} = O(1)
\]
and the result follows.

When $k=0$, we have $\widetilde{H}_{p,\alpha}(0,0)\in\mathbb{C}$.
If $\widetilde{H}_{p,\alpha}(0,0)=0$, then we set $x_2=0$ and choose $x_3$ as above, and then repeat this over $[x_3,1]$, with $\epsilon<|\widetilde{H}_{p,\alpha}(0,x_3)|$.
If $\widetilde{H}_{p,\alpha}(0,0)\neq0$, then we choose $\epsilon<|\widetilde{H}_{p,\alpha}(0,0)|$ and choose $x_1$, $x_2$, and $x_3$ as above.
We adjust the argument in a similar way for the case $k=p-1$.

The companion result \eqref{eqnSmallH} follows immediately upon application of Lemma~\ref{lemHstar}.
For \eqref{eqnSmallHX}, by Lemma~\ref{lemGXalpha} the function $H_{\mathbb{X},\alpha}(x)$ is almost surely continuous on $(0,1)$, and the proof is essentially the same.
\end{proof}

\begin{lemma}\label{lemLargeValsAtEndsInt}
Let $\alpha\in[0,1]$, and let $0<\epsilon<1/2$.
For primes $p\to\infty$, we have
\begin{equation}\label{eqnLargeH}
\frac{1}{p} \sum_{k=0}^{p-1} \left(\int_0^\epsilon + \int_{1-\epsilon}^1\right) \log\abs{H_{p,\alpha}(k,x)} \mathbbold{1}_{\abs{H_{p,\alpha}(k,x)}\geq\epsilon}\,dx \ll -\epsilon \log \epsilon.
\end{equation}
In addition, 
\begin{equation*}
\left(\int_0^\epsilon + \int_{1-\epsilon}^1\right) \mathbb{E}\left(\log\abs{H_{\mathbb{X},\alpha}(x)} \mathbbold{1}_{\abs{H_{\mathbb{X},\alpha}(x)}\geq\epsilon}\right)\,dx \ll -\epsilon \log \epsilon.
\end{equation*}
\end{lemma}

\begin{proof}
We consider the average over the integrals for $x\in[0,\epsilon]$ in \eqref{eqnLargeH}, the other cases are similar.
Let $L_{p,\alpha}(k,x) := x H_{p,\alpha}(k,x)$ so that
\begin{equation}\label{eqnHL}
\begin{split}
&\int_0^\epsilon \log\abs{H_{p,\alpha}(k,x)} \mathbbold{1}_{\abs{H_{p,\alpha}(k,x)}\geq\epsilon}\,dx =\\
&\quad
\int_0^\epsilon \log\abs{H_{p,\alpha}(k,x)} \mathbbold{1}_{\abs{H_{p,\alpha}(k,x)}\geq\epsilon} \mathbbold{1}_{\abs{L_{p,\alpha}(k,x)}>1}\,dx\\
&\qquad +
\int_0^\epsilon \log\abs{H_{p,\alpha}(k,x)} \mathbbold{1}_{\abs{H_{p,\alpha}(k,x)}\geq\epsilon} \mathbbold{1}_{\abs{L_{p,\alpha}(k,x)}\leq1}\,dx.
\end{split}
\end{equation}
One easily establishes that the second integral is $\ll -\epsilon\log \epsilon$, since
\begin{equation}\label{eqnHL2}
\begin{split}
\epsilon \log\epsilon
&\leq
\int_0^\epsilon \log\abs{H_{p,\alpha}(k,x)} \mathbbold{1}_{\abs{H_{p,\alpha}(k,x)}\geq\epsilon} \mathbbold{1}_{\abs{L_{p,\alpha}(k,x)}\leq1}\,dx\\
&\leq
\int_0^\epsilon \log\abs{L_{p,\alpha}(k,x)} \mathbbold{1}_{\abs{H_{p,\alpha}(k,x)}\geq\epsilon} \mathbbold{1}_{\abs{L_{p,\alpha}(k,x)}\leq1}\,dx + \int_0^\epsilon \abs{\log x}\,dx\\
&\leq
-c \epsilon\log \epsilon,
\end{split}
\end{equation}
for some positive constant $c$.
For the first integral, since $\abs{L_{p,\alpha}(k,x)}\leq\abs{H_{p,\alpha}(k,x)}$ for $x\in[0,1]$, we have
\begin{equation}\label{eqnHL1}
\begin{split}
&\int_0^\epsilon \log\abs{H_{p,\alpha}(k,x)} \mathbbold{1}_{\abs{H_{p,\alpha}(k,x)}\geq\epsilon} \mathbbold{1}_{\abs{L_{p,\alpha}(k,x)}>1}\,dx\\
&\qquad =
\int_0^\epsilon \log\abs{H_{p,\alpha}(k,x)} \mathbbold{1}_{\abs{L_{p,\alpha}(k,x)}>1}\,dx\\
&\qquad =
\int_0^\epsilon \log\abs{L_{p,\alpha}(k,x)} \mathbbold{1}_{\abs{L_{p,\alpha}(k,x)}>1}\,dx
+ \epsilon - \epsilon\log \epsilon.
\end{split}
\end{equation}
Next, writing $t$ for $\round{\alpha p}$ again, we have
\[
L_{p,\alpha}(k,x)
= x\hspace{-1ex}\sum_{\abs{j}<p/2} \leg{k-j}{p} \frac{\zeta_p^{jt}}{j+x} + O(1)
= x\hspace{-2ex}\sum_{0<\abs{j}<p/2} \leg{k-j}{p} \frac{\zeta_p^{jt}}{j} + O(1)
\]
uniformly for $0\leq x\leq 1/2$, so
\[
\max_{x\in[0,1/2]} \abs{L_{p,\alpha}(k,x)}
\leq
\abs{\sum_{0<\abs{j}<p/2}\leg{k-j}{p} \frac{\zeta_p^{jt}}{j}}^2 + C
\]
for some constant $C>1$.
Thus
\begin{equation}\label{eqnFirstIntL}
\begin{split}
0
&\leq
\int_0^\epsilon \log\abs{L_{p,\alpha}(k,x)} \mathbbold{1}_{\abs{L_{p,\alpha}(k,x)}>1}\,dx\\
&\leq
\epsilon\log\Biggl(\Big|\sum_{0<\abs{j}<p/2}\leg{k-j}{p} \frac{\zeta_p^{jt}}{j}\Big|^2 + C\Biggr).
\end{split}
\end{equation}
Using the concavity of the logarithm, we have
\[
\frac{1}{p} \sum_{k=0}^{p-1} \log\Biggl(\Big|\hspace{-1ex}\sum_{0<\abs{j}<p/2}\leg{k-j}{p} \frac{\zeta_p^{jt}}{j}\Big|^2 + C\Biggr)
\leq
\log\Biggl(\frac{1}{p} \sum_{k=0}^{p-1} \Big|\hspace{-1ex}\sum_{0<\abs{j}<p/2}\leg{k-j}{p} \frac{\zeta_p^{jt}}{j}\Big|^2 + C\Biggr),
\]
and
\begin{align*}
\frac{1}{p} \sum_{k=0}^{p-1} \Big|\hspace{-1ex}\sum_{0<\abs{j}<p/2}\leg{k-j}{p} \frac{\zeta_p^{jt}}{j}\Big|^2
&=
\sum_{\substack{0<\abs{j_1}<p/2\\0<\abs{j_2}<p/2}}
\frac{\zeta_p^{(j_1-j_2)t}}{p j_1 j_2} \sum_{k=0}^{p-1}
\leg{k-j_1}{p}\leg{k-j_2}{p}\\
&\leq
\frac{p-1}{p} \hspace{-1ex} \sum_{0<\abs{j}<p/2} \frac{1}{j^2} + \frac{1}{p} \sum_{\substack{0<\abs{j_1}<p/2\\0<\abs{j_2}<p/2\\j_1\neq j_2}} \frac{1}{\abs{j_1 j_2}}\\
&\ll
1 + \frac{(\log p)^2}{p} \ll 1,
\end{align*}
where we have again employed the result of Gauss that $\sum_{k=0}^{p-1} \leg{k-j_1}{p}\leg{k-j_2}{p}=-1$ for $p\geq3$ when $p\nmid(j_1-j_2)$.
Combining this with \eqref{eqnHL}, \eqref{eqnHL2}, \eqref{eqnHL1}, and \eqref{eqnFirstIntL} yields the result.
\end{proof}

To complete the proof of part~\ref{partMeasure}, we combine \eqref{eqnHmost} with Lemmas~\ref{lemSmallValsInt} and~\ref{lemLargeValsAtEndsInt} to obtain
\[
\lim_{p\to\infty} \frac{1}{p} \sum_{k=0}^{p-1} \int_0^1 \log \abs{H_{p,\alpha}(k,x)} \, dx
=
\int_0^1 \mathbb{E}\left(\log\abs{H_{\mathbb{X},\alpha}(x)}\right)\,dx + O\left(\epsilon^{0.24}\right)
\]
for sufficiently small $\epsilon>0$, and then let $\epsilon\to0^+$.
We then account for the additional term $\log(2\pi)$ from \eqref{eqnHdefn} and employ \eqref{eqnTurynToInterp} to obtain the statement.

\section{Proof of Corollary~\ref{corTuryn}}\label{secProofCor}

For parts~\ref{partPhiLW} and \ref{partLqLW}, we turn to the functions $G_{p,t}^{\pm}(k,x)$ from \eqref{eqnGptpm}.
Using \eqref{eqnTurynz}, \eqref{eqnGpt}, and \eqref{eqnFptpm}, we compute
\[
G_{p,t}^{\pm}(k,x) = G_{p,t}(k,x) \pm \frac{e(x(1-t/p))}{a(p)\sqrt{p}}.
\]
Note that the last term is independent of $k$.
Let $\alpha\in[0,1]$ be a fixed real number, and suppose $\varphi : C[0,1] \to \mathbb{C}$ is bounded and continuous.
Since $e(x(1-\alpha))/(a(p)\sqrt{p}) \to 0$ uniformly in $x$ as $p\to\infty$ and $\varphi$ is continuous, it follows that
\[
\lim_{p\to\infty} \frac{1}{p} \sum_{k=0}^{p-1} \varphi\left(G_{p,\round{\alpha p}}^{\pm}(k, x)\right) =
\lim_{p\to\infty} \frac{1}{p} \sum_{k=0}^{p-1} \varphi\left(G_{p,\round{\alpha p}}(k, x)\right)
= \mathbb{E}(\varphi(G_{\mathbb{X},\alpha}(x))).
\]
The conclusions of parts \ref{partPhi} and \ref{partLq} of Theorem~\ref{thmTuryn} therefore hold just as well for the functions $G_{p,\round{\alpha p}}^{\pm}(k,x)$ and the Littlewood polynomials $F_{p,\round{\alpha p}}^{\pm}(x)$ respectively.

For part \ref{partMeasureLW}, we define $H_{p,\alpha}^{\pm}(k,x)$ in an analogous way with respect to \eqref{eqnHdefn}, so that (with $t=\round{\alpha p}$)
\[
H_{p,\alpha}^{\pm}(k,x)
= H_{p,\alpha}(k,x) \pm \frac{\pi e\left(x\left(\frac{1}{2}-\frac{t}{p}\right)\right)}{a(p)\sin(\pi x) \sqrt{p}}
= H_{p,\alpha}(k,x) \pm \frac{\pi(\cot(\pi x)+i)}{a(p)\zeta_p^{tx}\sqrt{p}}.
\]
It follows that
\[
\frac{d}{dx}\left(H_{p,\alpha}^{\pm}(k,x) - H_{p,\alpha}(k,x)\right) = O\left(\frac{1}{x^2\sqrt{p}}\right)
\]
and
\[
\frac{d^2}{dx^2}\left(H_{p,\alpha}^{\pm}(k,x) - H_{p,\alpha}(k,x)\right) = O\left(\frac{1}{x^3\sqrt{p}}\right)
\]
for $0\leq k<p$.
Each of these is a factor of $\sqrt{p}$ smaller than any term in the series employed to estimate the first and second derivatives of $H_{p,\alpha}(k,x)$ in the proof of Lemma~\ref{lemSmallValsInt}, so for sufficiently large $p$ the additional term in $H_{p,\alpha}^{\pm}(k,x)$ will not disturb the estimates in Section~\ref{subsecPartMeasure} or the results obtained there, except possibly when $x$ is very small relative to $p$.
However, the measure of this exceptional set is so small that these values cannot affect the estimates when $p$ is large enough: for fixed $\delta>0$ the contribution from the left side of \eqref{eqnsmallwtH} over $0<x<1/p^{1/4-\delta}$ is at most $\abs{\log\epsilon}/p^{1/4-\delta}$.
When $\delta=1/100$, this is $O(\epsilon^{0.24})$ provided $p>\epsilon^{-1}\abs{\log\epsilon}^{25/6}$.

\section{Calculations}\label{secCalculations}

\subsection{Mahler measure}\label{subsecMahlerMeas}
We use Theorem~\ref{thmTuryn}\ref{partMeasure} to estimate the asymptotic value of the normalized Mahler measure of the Turyn polynomials.
Since
\[
\int_0^1 \log\abs{\frac{e(x)-1}{2\pi i}}\,dx
= -\log(2\pi),
\]
we need to investigate the value of
\begin{align}\label{eqnLambda0}
\lambda_0(\alpha)
&= \lim_{J\to\infty} \frac{1}{2^{2J+1}} \sum_{\substack{\delta_j=\pm1\\\abs{j}\leq J}} \int_0^1 \log \Biggl\lvert \sum_{\abs{j}\leq J} \frac{\delta_j e(\alpha j)}{j+x} \Biggr\rvert \, dx\notag\\
&= \lim_{J\to\infty} \frac{1}{4^J} \sum_{\substack{\delta_j=\pm1\\0<\abs{j}\leq J}} \int_0^1 \log \Biggl\lvert \frac{1}{x} + \sum_{0<\abs{j}\leq J} \frac{\delta_j e(\alpha j)}{j+x} \Biggr\rvert \, dx,
\end{align}
since then
\[
\kappa_0(\alpha) = \frac{e^{\lambda_0(\alpha)}}{2\pi}.
\]
Let $\lambda_0^J(\alpha)$ denote the expression inside the limit in \eqref{eqnLambda0}, and let $\kappa_0^J(\alpha) = e^{\lambda_0^J(\alpha)}/2\pi$.
Since
\[
\frac{\delta_j e(\alpha j)}{x+j} + \frac{\delta_{-j} e(-\alpha j)}{x-j} = \frac{(\delta_je(\alpha j)+\delta_{-j}e(-\alpha j))x-(\delta_je(\alpha j)-\delta_{-j}e(-\alpha j))j}{x^2-j^2},
\]
we have
\[
\lambda_0^J(\alpha) = \frac{1}{4^J} \hspace{-2ex} \sum_{\substack{1\leq j\leq J\\\beta_{j,\alpha}(x)\in B_{j,\alpha,x}}} \hspace{-2ex} \int_0^1 \log \Biggl\lvert \frac{1}{x} + \sum_{j=1}^J \frac{2\beta_{j,\alpha}(x)}{x^2-j^2} \Biggr\rvert \, dx
\]
where
\[
B_{j,\alpha,x} = \{\pm(x\cos(2\pi \alpha j) - i j\sin(2\pi \alpha j)), \pm(i x\sin(2\pi \alpha j) - j\cos(2\pi \alpha j))\}.
\]
It follows that
\[
\lambda_0^J(\alpha) = \lambda_0^J\left(\frac{1}{2}\pm \alpha\right),
\]
so we may assume $0\leq \alpha \leq 1/4$.
In addition, we compute that
\[
\int_0^1 \biggl(\log x + \sum_{j=1}^J \log (j^2-x^2)\biggr) \, dx = J\log J + (J+1)\log(J+1) - 2J - 1,
\]
so
\begin{equation}\label{eqnTurynComputeAlpha}
\begin{split}
\lambda_0^J(\alpha)
&= 4^{-J} \hspace{-4ex} \sum_{\substack{1\leq j\leq J\\\beta_{j,\alpha}(x)\in B_{j,\alpha,x}}} \hspace{-2ex} \int_0^1 \log \biggl\lvert \prod_{\ell=1}^J (x^2-\ell^2) + 2x\beta_{j,\alpha}(x)\prod_{\substack{1\leq \ell\leq J\\\ell\neq j}} (x^2-\ell^2) \biggr\rvert \, dx\\
&\qquad - J\log J - (J+1)\log(J+1) + 2J + 1.
\end{split}
\end{equation}
We employed \eqref{eqnTurynComputeAlpha} to compute $\lambda_0^J(\alpha)$, and hence $\kappa_0^J(\alpha)$, with $\alpha = a/1600$ for $0\leq a\leq 400$, and with $J\leq 14$, using Julia \cite{Julia} to perform the computations.
Figure~\ref{figAll14} shows plots for these $\kappa_0^J(\alpha)$ for $J=5$, $8$, and $14$.
These are distinguishable at the displayed scale only when $\alpha$ is rather small, but for fixed $\alpha$ the values increase in $J$, so the $J=5$ case is the lowest plot and $J=14$ is the highest.
Figure~\ref{figGain14} exhibits the values of $\kappa_0^{14}(\alpha)-\kappa_0^{13}(\alpha)$ at each sampled value $\alpha$, to show the size of these increases.

\begin{figure}[tbhp]
\begin{center}
\includegraphics[width=\figwidth]{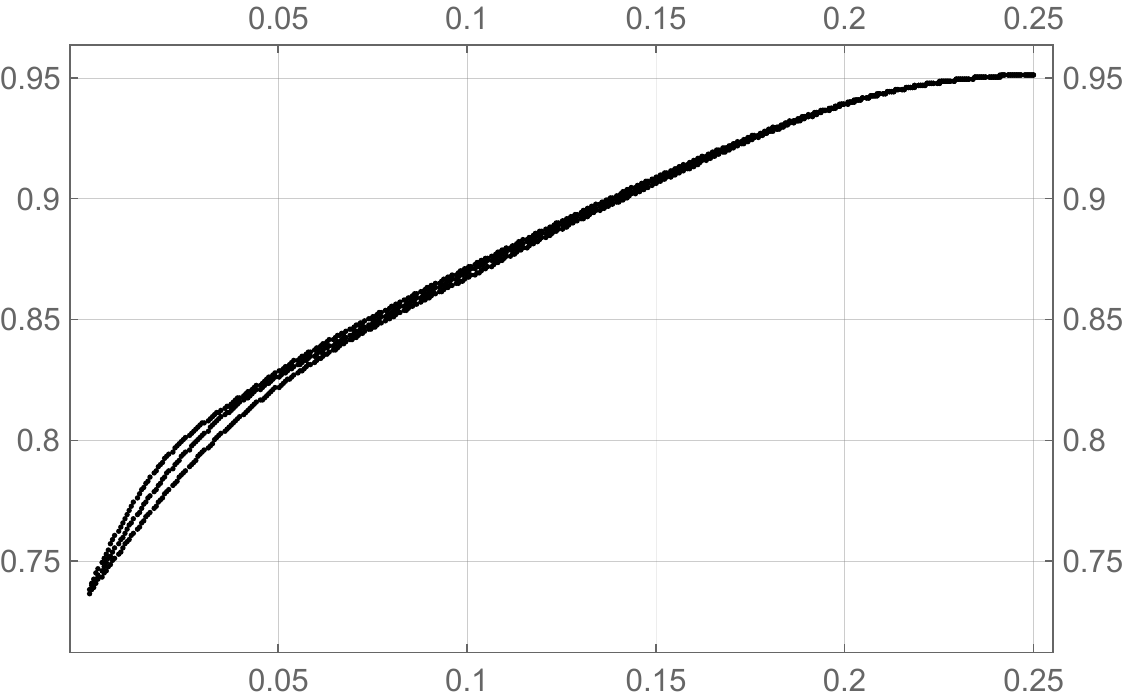}
\end{center}
\caption{$\kappa_0^5(\alpha)$ (lowest), $\kappa_0^8(\alpha)$, and $\kappa_0^{14}(\alpha)$ (highest) for $0\leq\alpha\leq1/4$.}\label{figAll14}
\end{figure}

\begin{figure}[tbhp]
\begin{center}
\includegraphics[width=\figwidth]{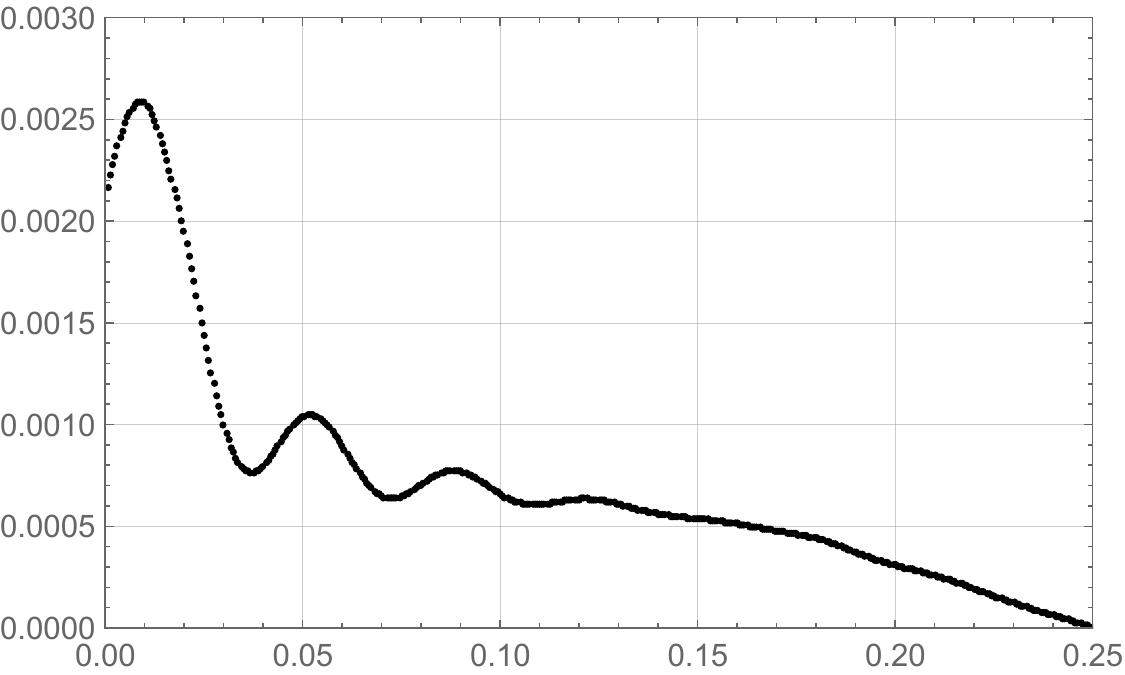}
\end{center}
\caption{$\kappa_0^{14}(\alpha)-\kappa_0^{13}(\alpha)$ for $0\leq\alpha\leq 1/4$.}\label{figGain14}
\end{figure}

The data suggests that $\kappa_0(\alpha)$ is increasing for $0\leq\alpha\leq1/4$, and motivates further study for the case $\alpha=1/4$.
Here, the allowable values $\beta_{j,1/4}(x)$ are simpler: 
\[
B_{2j,1/4,x} = \{\pm x, \pm j\}, \quad
B_{2j+1,1/4,x} = \{\pm i x, \pm i j\}.
\]
We computed $\kappa_0^J(1/4)$ for $J\leq 18$, and list the results in Table~\ref{tableMeasLqData}.
The same table displays the values we determined for $\kappa_0^J(0)$ (so the case of Fekete polynomials), again for $J\leq 18$.
Note in this case $B_{j,0,x} = \{\pm x, \pm j\}$.

The values $\kappa_0^J(0)$ and $\kappa_0^J(1/4)$ from Table~\ref{tableMeasLqData} are plotted in Figures~\ref{figFekByJ} and~\ref{figTurynByJ} respectively, where these values are shown as solid disks for $J\leq18$.
In each case we also plot the best fitting curve of the form $r+s/J+t/J^2$ over $7\leq J\leq18$.
For the $\alpha=0$ and $\alpha=1/4$ cases these curves are respectively
\[
0.73990 - \frac{0.01502}{J} - \frac{0.00453}{J^2},
\quad
0.95114 - \frac{0.000926}{J} + \frac{0.00272}{J^2}.
\]
While calculating a precise value for $\kappa_0^J(0)$ or $\kappa_0^J(1/4)$ for $J>18$ was not practical, we estimated the values for some larger $J$ by computing a sizable random sample of the $4^J$ integrals.
For each $19\leq J\leq 40$, we computed $2^{28}$ randomly selected integrals from the expression in \eqref{eqnTurynComputeAlpha}, and computed their mean value.
We show the resulting data points in Figures~\ref{figFekByJ} and~\ref{figTurynByJ} respectively using $\times$ symbols.

Thus, it appears that $\kappa_0(0) = 0.740$ and $\kappa_0(1/4) = 0.951$, to three decimal places of precision.
These are numerical estimates, without proof of accuracy.

\begin{figure}[tbhp]
\begin{center}
\includegraphics[width=\figwidth]{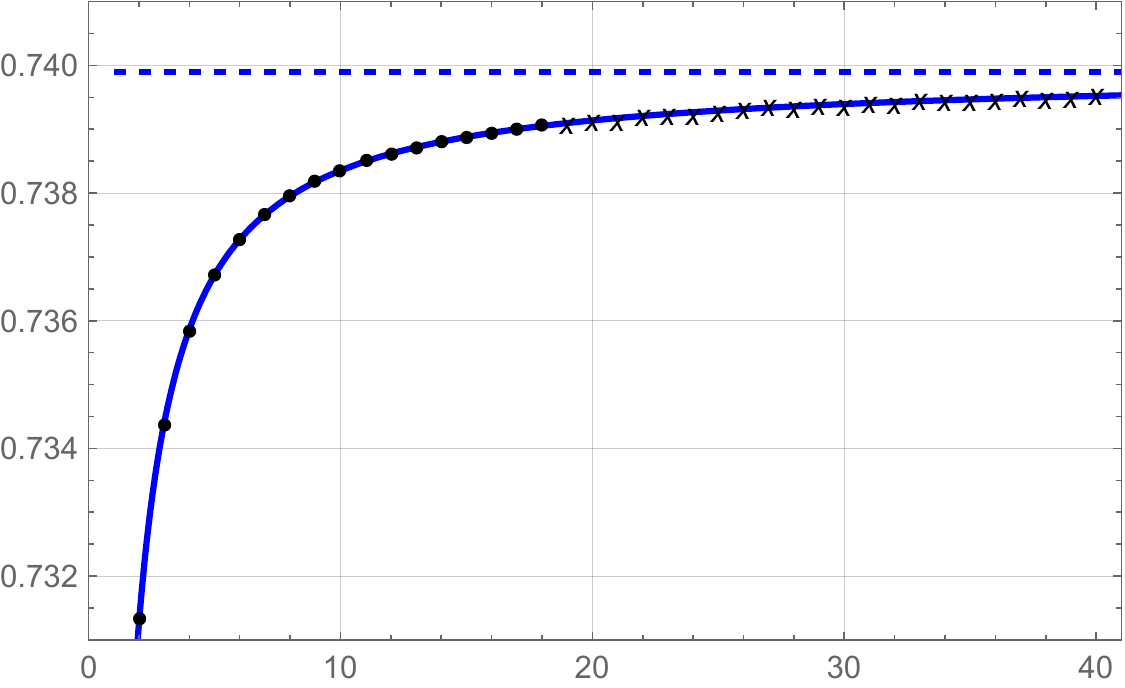}
\end{center}
\caption{$\kappa_0^J(0)$ for $J\leq18$ (solid points), their interpolating curve and its asymptotic value (solid curve and dashed line), and sampling estimates for $\kappa_0^J(0)$ for $19\leq J\leq40$ ($\times$ points).}\label{figFekByJ}
\end{figure}

\begin{figure}[tbhp]
\begin{center}
\includegraphics[width=\figwidth]{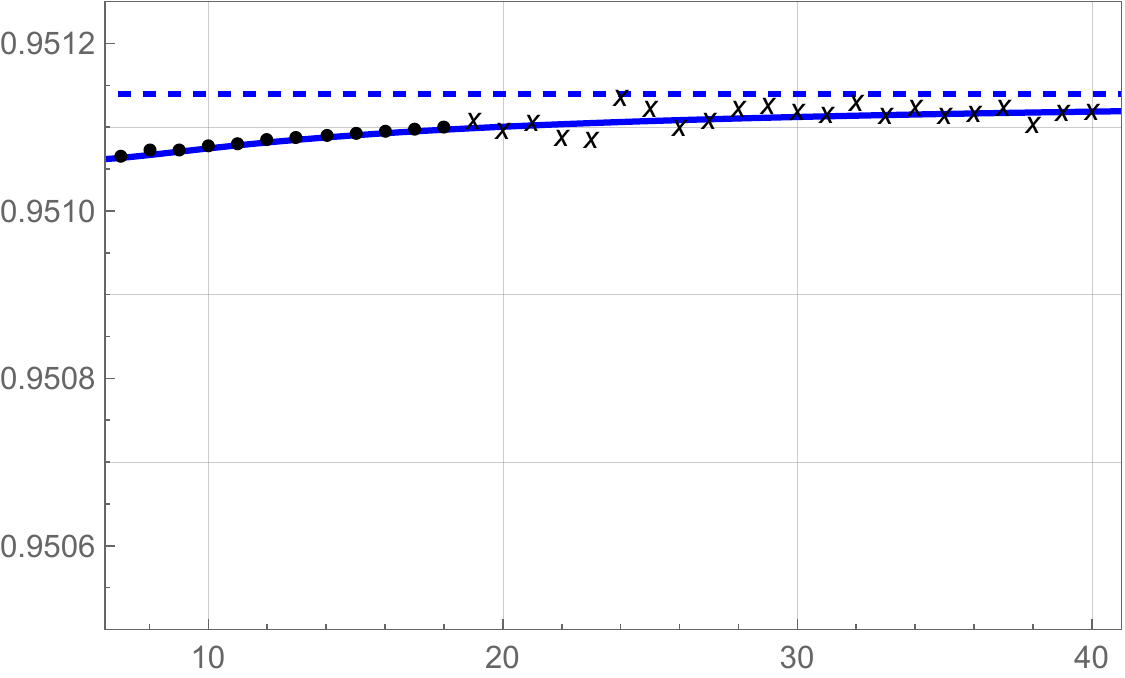}
\end{center}
\caption{$\kappa_0^J(1/4)$ for $J\leq18$ (solid points), their interpolating curve and its asymptotic value (solid curve and dashed line), and sampling estimates for $\kappa_0^J(1/4)$ for $19\leq J\leq40$ ($\times$ points).}\label{figTurynByJ}
\end{figure}

Finally, we present some calculations related to Corollary~\ref{corTuryn} regarding the companion Littlewood polynomials $F_{p.\round{t/4}}^{\pm}$ constructed from the Turyn polynomials with relative shift $\alpha=1/4$.
For each prime $p<2000$, we compute the normalized Mahler measure of this Turyn polynomial and each of its two companion Littlewood polynomials, using $\sqrt{p}$ as the normalizing factor for all three polynomials.
At each prime we record the two values
\begin{equation}\label{eqnMeasTLW}
\frac{1}{\sqrt{p}}\left(M(F_{p,\round{t/4}}^+) - M(F_{p,\round{t/4}})\right), \quad
\frac{1}{\sqrt{p}}\left(M(F_{p,\round{t/4}}^-) - M(F_{p,\round{t/4}})\right).
\end{equation}
These numbers show reasonably good convergence toward $0$ as $p$ grows large, as expected from the corollary, and are displayed in Figure~\ref{figTurynToLW}.

\begin{figure}[tbhp]
\begin{center}
\includegraphics[width=\figwidth]{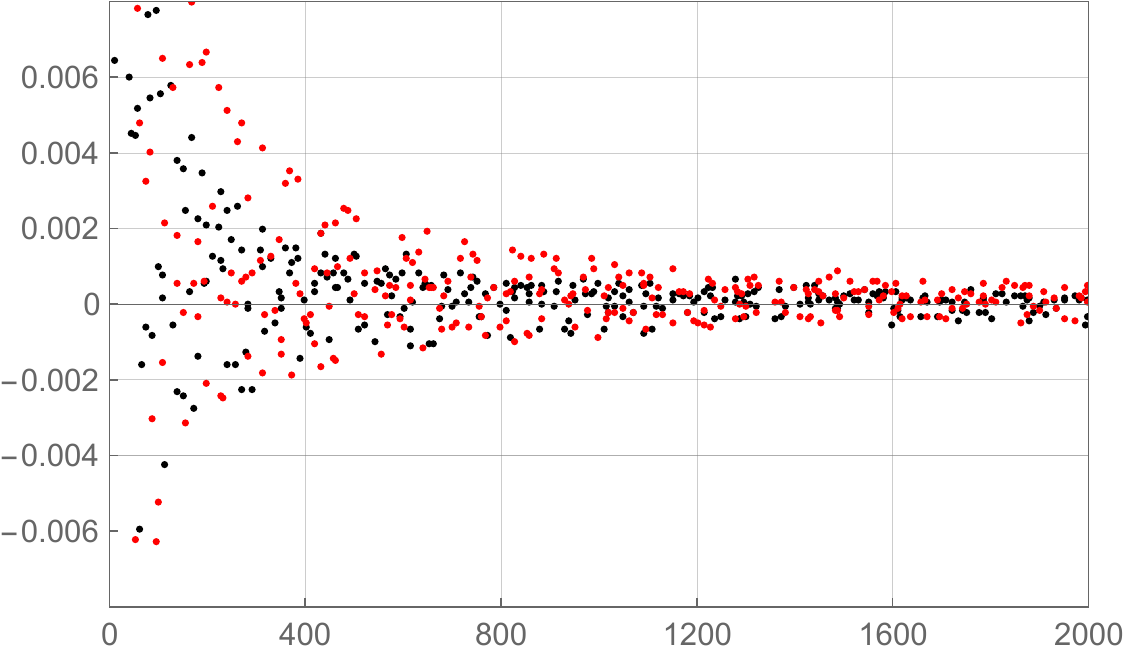}
\end{center}
\caption{Difference between the normalized Mahler measure of the Turyn polynomial with relative shift $\alpha=1/4$, and its two companion Littlewood polynomials from Corollary~\ref{corTuryn}, for primes $p<2000$, using \eqref{eqnMeasTLW}.
The black points show the result using $F_{p,\round{p/4}}^+(x)$; the red ones illustrate those from $F_{p,\round{p/4}}^-(x)$.}\label{figTurynToLW}
\end{figure}

\subsection{$L_q$ norms}\label{subsecLq}

We can estimate the normalized $L_q$ norms of the Fekete and Turyn polynomials in the same way.
For $q\neq0$, let
\[
\lambda_q^J(\alpha) = \frac{1}{4^J} \hspace{-2ex} \sum_{\substack{1\leq j\leq J\\\beta_{j,\alpha}(x)\in B_{j,\alpha,x}}} \hspace{-2ex} \int_0^1 \Biggl\lvert \sin(\pi x) \Biggl(\frac{1}{x} + \sum_{j=1}^J \frac{2\beta_{j,\alpha}(x)}{x^2-j^2}\Biggr) \Biggr\rvert^q \, dx,
\]
so that our estimate for the normalized $L_q$ norm for the Turyn polynomials with shift value $\alpha$, relative to the parameter $J$, is
\[
\kappa_q^J(\alpha) = \frac{\left(\lambda_q^J(\alpha)\right)^{1/q}}{\pi}.
\]
Motivated by Newman's problem, we use this strategy with $J\leq18$ to estimate the normalized $L_1$ norm of the Turyn polynomials with shift $\alpha=1/4$.
For symmetry we computed values for the $q=3$ case in the same range.
Our results are displayed in Table~\ref{tableMeasLqData}.
The best fits for the data in this table with $5\leq J\leq18$ to curves of the form $r+s/J+t/J^2$ are given by
\begin{equation*}\label{eqnFitTurL1L3}
\begin{split}
\kappa_1^J(1/4) &\approx 0.97757 - \frac{0.02849}{J} + \frac{0.01132}{J^2},\\
\kappa_3^J(1/4) &\approx 1.02022 - \frac{0.06969}{J} + \frac{0.02377}{J^2},
\end{split}
\end{equation*}
and the former provides our estimate \eqref{eqnTurynL1}.
In the same way, the best fits for the normalized $L_1$ and $L_3$ norms of the Fekete polynomials are
\begin{equation*}\label{eqnFitFekL1L3}
\begin{split}
\kappa_1^J(0) &\approx 0.90482 - \frac{0.02962}{J} + \frac{0.00896}{J^2},\\
\kappa_3^J(0) &\approx 1.07285 - \frac{0.07444}{J} + \frac{0.02548}{J^2}.
\end{split}
\end{equation*}

\begin{table}[tbhp]\small
\caption{Estimating the normalized Mahler measure, $L_1$ norm, and $L_3$ norm of the Fekete polynomials  ($\alpha=0$) and the Turyn polynomials with $\alpha=1/4$.}\label{tableMeasLqData}
\begin{tabular}{|c|cc|cc|cc|}\hline
\TS\BS$J$ & $\kappa_0^J(0)$ & $\kappa_0^J(1/4)$ & $\kappa_1^J(0)$ & $\kappa_1^J(1/4)$ & $\kappa_3^J(0)$ & $\kappa_3^J(1/4)$\\\hline
$1$ & $0.72251765$ & $0.95073546$ & $0.88137013$ & $0.95679218$ & $1.01651442$ & $0.96780395$\\
$2$ & $0.73134619$ & $0.95138014$ & $0.89194482$ & $0.96594791$ & $1.04123765$ & $0.99048801$\\
$3$ & $0.73437217$ & $0.95105638$ & $0.89588907$ & $0.96925694$ & $1.05072786$ & $0.99951092$\\
$4$ & $0.73584586$ & $0.95110175$ & $0.89796873$ & $0.97115394$ & $1.05580391$ & $1.00424633$\\
$5$ & $0.73671019$ & $0.95106467$ & $0.89925856$ & $0.97231756$ & $1.05897632$ & $1.00723295$\\
$6$ & $0.73727369$ & $0.95107405$ & $0.90013817$ & $0.97313514$ & $1.06115034$ & $1.00926538$\\
$7$ & $0.73766376$ & $0.95106484$ & $0.90077701$ & $0.97372578$ & $1.06273445$ & $1.01075255$\\
$8$ & $0.73795637$ & $0.95107230$ & $0.90126252$ & $0.97418145$ & $1.06394060$ & $1.01188075$\\
$9$ & $0.73817947$ & $0.95107306$ & $0.90164404$ & $0.97453830$ & $1.06488988$ & $1.01277084$\\
$10$ & $0.73835574$ & $0.95107808$ & $0.90195181$ & $0.97482882$ & $1.06565657$ & $1.01348812$\\
$11$ & $0.73849913$ & $0.95108034$ & $0.90220539$ & $0.97506762$ & $1.06628881$ & $1.01408054$\\
$12$ & $0.73861926$ & $0.95108454$ & $0.90241795$ & $0.97526901$ & $1.06681916$ & $1.01457672$\\
$13$ & $0.73872036$ & $0.95108715$ & $0.90259872$ & $0.97543999$ & $1.06727042$ & $1.01499939$\\
$14$ & $0.73880677$ & $0.95109059$ & $0.90275433$ & $0.97558782$ & $1.06765908$ & $1.01536301$\\
$15$ & $0.73888148$ & $0.95109306$ & $0.90288970$ & $0.97571625$ & $1.06799733$ & $1.01567974$\\
$16$ & $0.73894672$ & $0.95109592$ & $0.90300854$ & $0.97582938$ & $1.06829438$ & $1.01595766$\\
$17$ & $0.73900419$ & $0.95109813$ & $0.90311371$ & $0.97592939$ & $1.06855734$ & $1.01620383$\\
$18$ & $0.73905520$ & $0.95110053$ & $0.90320744$ & $0.97601875$ & $1.06879176$ & $1.01642314$\\\hline
\end{tabular}
\end{table}

\section{Generalized Turyn polynomials}\label{secFuture}

We close with a remark concerning potential future research.
A further generalization of the Fekete polynomials arises in some prior work on $L_4$ norms and Golay's merit factor problem from Section~\ref{secIntroduction}.
In the \textit{generalized Turyn polynomials} $F_{p,t,d}(x)$, in addition to the rotation parameter $t$, one introduces a second integer parameter $d$ that adjusts the degree of the resulting polynomial, causing truncation when $0\leq d<p$, and periodic extension when $d\geq p$:
\[
F_{p,t,d}(x) = \sum_{j=0}^{d} \leg{j + t}{p} x^j.
\]
Following computational studies by A.~Kirilusha and G.~Narayanaswamy in 1999, and then by Borwein, Choi, and Jedwab in 2004 \cite{BCJ04}, in 2013 Jedwab, Katz, and Schmidt \cite{JKS13} proved that normalized $L_4$ norms smaller than the value $(7/6)^{1/4}=1.039289\ldots$ achieved by $F_{p,\round{p/4}}$ are attained by certain generalized Turyn polynomials.
They showed that the optimal value for this family, $1.037282\ldots<(22/19)^{1/4}$, occurs when $t\approx0.221p$ and $d\approx 1.058p$, where the constants are particular algebraic numbers of degree~$3$.
It seems likely that these same polynomials would also possess large normalized Mahler measure, potentially larger than the $0.951$ value attained by the Turyn polynomials studied here with $t=\round{p/4}$, as well as larger normalized $L_1$ norms for Newman's problem.
Adapting the methods of the present paper to this family would introduce additional complications, such as incomplete character sums.\footnote{During the refereeing process, Bober, Klurman, and Shala announced \cite{BKS26}, which develops a new approach for studying distributions of mixed character sums, and allows analysis of these generalized Turyn polynomials.
Using this family, they find further improvements in both Mahler's and Newman's problems,
attaining approximately $0.954$ and $0.9783$ respectively.}

\section*{Acknowledgments}

I thank Stephen Choi, Oleksiy Klurman, Youness Lamzouri, and Marc Munsch for beneficial conversations and correspondence.
In particular, I thank Oleksiy Klurman for indicating the method of the proof of Corollary~\ref{corTuryn}.
I also thank the referee for their careful reading and helpful comments.

\bibliographystyle{amsplain}

\begin{bibdiv}
\begin{biblist}

\bib{BN73}{article}{
   author={Beller, E.},
   author={Newman, D. J.},
   title={An extremal problem for the geometric mean of polynomials},
   journal={Proc. Amer. Math. Soc.},
   volume={39},
   date={1973},
   pages={313--317},
   issn={0002-9939},
   review={\MR{316686}},
   doi={10.2307/2039638},
}

\bib{Julia}{article}{
    author={Bezanson, Jeff},
    author={Edelman, Alan},
    author={Shah, Viral B.},
    title={Julia: A fresh approach to numerical computing},
    journal={SIAM Review},
    volume={59},
    number={1},
    pages={65--98},
    year={2017},
    publisher={SIAM},
    doi={10.1137/141000671},
    url={https://epubs.siam.org/doi/10.1137/141000671}
}

\bib{BKS26}{article}{
   author={Bober, Jonathan W.},
   author={Klurman, Oleksiy},
   author={Shala, Besfort},
   note={arXiv:2510.06161v2 [math.NT], 16 Mar 2026},
   title={Distribution of mixed character sums and extremal problems for Littlewood polynomials},
   pages={37 pp.},
}

\bib{Borwein02}{book}{
   author={Borwein, Peter},
   title={Computational excursions in analysis and number theory},
   series={CMS Books in Mathematics/Ouvrages de Math\'{e}matiques de la SMC},
   volume={10},
   publisher={Springer-Verlag, New York},
   date={2002},
   pages={x+220},
   isbn={0-387-95444-9},
   review={\MR{1912495}},
   doi={10.1007/978-0-387-21652-2},
}

\bib{BC02}{article}{
   author={Borwein, Peter},
   author={Choi, Kwok-Kwong Stephen},
   title={Explicit merit factor formulae for Fekete and Turyn polynomials},
   journal={Trans. Amer. Math. Soc.},
   volume={354},
   date={2002},
   number={1},
   pages={219--234},
   issn={0002-9947},
   review={\MR{1859033}},
   doi={10.1090/S0002-9947-01-02859-8},
}

\bib{BCJ04}{article}{
   author={Borwein, Peter},
   author={Choi, Kwok-Kwong Stephen},
   author={Jedwab, Jonathan},
   title={Binary sequences with merit factor greater than 6.34},
   journal={IEEE Trans. Inform. Theory},
   volume={50},
   date={2004},
   number={12},
   pages={3234--3249},
   issn={0018-9448},
   review={\MR{2103494}},
   doi={10.1109/TIT.2004.838341},
}

\bib{BDM07}{article}{
   author={Borwein, Peter},
   author={Dobrowolski, Edward},
   author={Mossinghoff, Michael J.},
   title={Lehmer's problem for polynomials with odd coefficients},
   journal={Ann. of Math. (2)},
   volume={166},
   date={2007},
   number={2},
   pages={347--366},
   issn={0003-486X},
   review={\MR{2373144}},
   doi={10.4007/annals.2007.166.347},
}

\bib{BM08}{article}{
   author={Borwein, Peter},
   author={Mossinghoff, Michael J.},
   title={Barker sequences and flat polynomials},
   conference={
      title={Number theory and polynomials},
   },
   book={
      series={London Math. Soc. Lecture Note Ser.},
      volume={352},
      publisher={Cambridge Univ. Press, Cambridge},
   },
   date={2008},
   pages={71--88},
   review={\MR{2428516}},
   doi={10.1017/CBO9780511721274.007},
}

\bib{CE14}{article}{
   author={Choi, Stephen},
   author={Erd\'{e}lyi, Tam\'{a}s},
   title={Average Mahler's measure and $L_p$ norms of Littlewood polynomials},
   journal={Proc. Amer. Math. Soc. Ser. B},
   volume={1},
   date={2014},
   pages={105--120},
   review={\MR{3272724}},
   doi={10.1090/S2330-1511-2014-00013-4},
}

\bib{CE15}{article}{
   author={Choi, Stephen},
   author={Erd\'{e}lyi, Tam\'{a}s},
   title={Sums of monomials with large Mahler measure},
   journal={J. Approx. Theory},
   volume={197},
   date={2015},
   pages={49--61},
   issn={0021-9045},
   review={\MR{3351539}},
   doi={10.1016/j.jat.2014.01.003},
}

\bib{CM11}{article}{
   author={Choi, Kwok-Kwong Stephen},
   author={Mossinghoff, Michael J.},
   title={Average Mahler's measure and $L_p$ norms of unimodular polynomials},
   journal={Pacific J. Math.},
   volume={252},
   date={2011},
   number={1},
   pages={31--50},
   issn={0030-8730},
   review={\MR{2862140}},
   doi={10.2140/pjm.2011.252.31},
}

\bib{CGPS00}{article}{
   author={Conrey, B.},
   author={Granville, A.},
   author={Poonen, B.},
   author={Soundararajan, K.},
   title={Zeros of Fekete polynomials},
   journal={Ann. Inst. Fourier (Grenoble)},
   volume={50},
   date={2000},
   number={3},
   pages={865--889},
   issn={0373-0956},
   review={\MR{1779897}},
}

\bib{DM05}{article}{
   author={Dubickas, Art\={u}ras},
   author={Mossinghoff, Michael J.},
   title={Auxiliary polynomials for some problems regarding Mahler's
   measure},
   journal={Acta Arith.},
   volume={119},
   date={2005},
   number={1},
   pages={65--79},
   issn={0065-1036},
   review={\MR{2163518}},
   doi={10.4064/aa119-1-5},
}

\bib{Erdelyi18}{article}{
   author={Erd\'{e}lyi, Tam\'{a}s},
   title={Improved lower bound for the Mahler measure of the Fekete polynomials},
   journal={Constr. Approx.},
   volume={48},
   date={2018},
   number={2},
   pages={283--299},
   issn={0176-4276},
   review={\MR{3848040}},
   doi={10.1007/s00365-017-9398-y},
}

\bib{Erdelyi20}{article}{
   author={Erd\'{e}lyi, Tam\'{a}s},
   title={The asymptotic value of the Mahler measure of the Rudin-Shapiro polynomials},
   journal={J. Anal. Math.},
   volume={142},
   date={2020},
   number={2},
   pages={521--537},
   issn={0021-7670},
   review={\MR{4205789}},
   doi={10.1007/s11854-020-0142-3},
}

\bib{EL07}{article}{
   author={Erd\'{e}lyi, T.},
   author={Lubinsky, D. S.},
   title={Large sieve inequalities via subharmonic methods and the Mahler measure of the Fekete polynomials},
   journal={Canad. J. Math.},
   volume={59},
   date={2007},
   number={4},
   pages={730--741},
   issn={0008-414X},
   review={\MR{2338232}},
   doi={10.4153/CJM-2007-032-x},
}

\bib{Fielding70}{article}{
   author={Fielding, G. T.},
   title={The expected value of the integral around the unit circle of a certain class of polynomials},
   journal={Bull. London Math. Soc.},
   volume={2},
   date={1970},
   pages={301--306},
   issn={0024-6093},
   review={\MR{280689}},
   doi={10.1112/blms/2.3.301},
}

\bib{GIMPW10}{article}{
   author={Garza, J.},
   author={Ishak, M. I. M.},
   author={Mossinghoff, M. J.},
   author={Pinner, C. G.},
   author={Wiles, B.},
   title={Heights of roots of polynomials with odd coefficients},
   journal={J. Th\'{e}or. Nombres Bordeaux},
   volume={22},
   date={2010},
   number={2},
   pages={369--381},
   issn={1246-7405},
   review={\MR{2769068}},
}

\bib{Golay83}{article}{
   author={Golay, M.~J.~E.},
   title={The merit factor of Legendre sequences},
   journal={IEEE Trans. Inform. Theory},
   volume={29},
   date={1983},
   pages={934--936},
}

\bib{GS17}{article}{
   author={G\"{u}nther, Christian},
   author={Schmidt, Kai-Uwe},
   title={$L^q$ norms of Fekete and related polynomials},
   journal={Canad. J. Math.},
   volume={69},
   date={2017},
   number={4},
   pages={807--825},
   issn={0008-414X},
   review={\MR{3679696}},
   doi={10.4153/CJM-2016-023-4},
}

\bib{HJ88}{article}{
   author={H{\o}holdt, T.},
   author={Jensen, H.},
   title={Determination of the merit factor of Legendre sequences},
   journal={IEEE Trans. Inform. Theory},
   volume={34},
   date={1988},
   pages={161--164},
}

\bib{Jedwab05}{article}{
   author={Jedwab, Jonathan},
   title={A survey of the merit factor problem for binary sequences},
   book={
      title={Proceedings of Sequences and Their Applications},
      series={Lecture Notes in Comput. Sci.},
      volume={3486},
      publisher={Springer, New York},
   },
   date={2005},
   pages={30--55},
}

\bib{JKS13}{article}{
   author={Jedwab, Jonathan},
   author={Katz, Daniel J.},
   author={Schmidt, Kai-Uwe},
   title={Littlewood polynomials with small $L^4$ norm},
   journal={Adv. Math.},
   volume={241},
   date={2013},
   pages={127--136},
   issn={0001-8708},
   review={\MR{3053707}},
   doi={10.1016/j.aim.2013.03.015},
}

\bib{Katz18}{article}{
   author={Katz, Daniel J.},
   title={Sequences with low correlation},
   conference={
      title={Arithmetic of finite fields},
   },
   book={
      series={Lecture Notes in Comput. Sci.},
      volume={11321},
      publisher={Springer, Cham},
   },
   date={2018},
   pages={149--172},
   review={\MR{3905074}},
   doi={10.1007/978-3-030-05153-2\_8},
}

\bib{KLM23}{article}{
   author={Klurman, Oleksiy},
   author={Lamzouri, Youness},
   author={Munsch, Marc},
   note={arXiv:2306.07156 [math.NT], 12 Jun 2023},
   title={$L_q$ norms and Mahler measure of Fekete polynomials},
   pages={24 pp.},
}

\bib{Krylov02}{book}{
   author={Krylov, N.~V.},
   title={Introduction to the theory of random processes},
   series={Grad. Stud. Math.},
   volume={43},
   publisher={American Mathematical Society, Providence, RI},
   date={2002},
   pages={xii+230},
   isbn={0-8218-2985-8},
   review={\MR{1885884}},
   doi={10.1090/gsm/043},
}

\bib{Lehmer33}{article}{
   author={Lehmer, D.~H.},
   title={Factorization of certain cyclotomic functions},
   journal={Ann. of Math. (2)},
   volume={34},
   date={1933},
   number={3},
   pages={461--479},
   issn={0003-486X},
   review={\MR{1503118}},
   doi={10.2307/1968172},
}

\bib{Littlewood61}{article}{
   author={Littlewood, J.~E.},
   title={On the mean values of certain trigonometric polynomials},
   journal={J. London Math. Soc.},
   volume={36},
   date={1961},
   pages={307--334},
   issn={0024-6107},
   review={\MR{141934}},
   doi={10.1112/jlms/s1-36.1.307},
}

\bib{Mahler63}{article}{
   author={Mahler, Kurt},
   title={On two extremum properties of polynomials},
   journal={Illinois J. Math.},
   volume={7},
   date={1963},
   pages={681--701},
   issn={0019-2082},
   review={\MR{156950}},
}

\bib{Montgomery80}{article}{
   author={Montgomery, Hugh L.},
   title={An exponential polynomial formed with the Legendre symbol},
   journal={Acta Arith.},
   volume={37},
   date={1980},
   pages={375--380},
   issn={0065-1036},
   review={\MR{598890}},
   doi={10.4064/aa-37-1-375-380},
}

\bib{Newman60}{article}{
   author={Newman, D.~J.},
   title={Norms of polynomials},
   date={1960},
   journal={Amer. Math. Monthly},
   volume={67},
   pages={778\ndash 779},
   review={\MR{0125205}},
}

\bib{Newman65}{article}{
   author={Newman, D.~J.},
   title={An {$L\sp{1}$} extremal problem for polynomials},
   date={1965},
   journal={Proc. Amer. Math. Soc.},
   volume={16},
   pages={1287\ndash 1290},
   review={\MR{0185119}},
}

\bib{Pritsker08}{article}{
   author={Pritsker, Igor E.},
   title={Polynomial inequalities, Mahler's measure, and multipliers},
   conference={
      title={Number theory and polynomials},
   },
   book={
      series={London Math. Soc. Lecture Note Ser.},
      volume={352},
      publisher={Cambridge Univ. Press, Cambridge},
   },
   date={2008},
   pages={255--276},
   review={\MR{2428526}},
   doi={10.1017/CBO9780511721274.017},
}

\bib{Schmidt16}{article}{
   author={Schmidt, Kai-Uwe},
   title={Sequences with small correlation},
   journal={Des. Codes Cryptogr.},
   volume={78},
   date={2016},
   number={1},
   pages={237--267},
   issn={0925-1022},
   review={\MR{3440231}},
   doi={10.1007/s10623-015-0154-7},
}

\bib{Schmidt76}{book}{
   author={Schmidt, Wolfgang M.},
   title={Equations over finite fields: an elementary approach},
   series={Lecture Notes in Math.},
   volume={536},
   publisher={Springer, Berlin},
   date={1976},
   pages={ix+276},
   isbn={3-540-07855-X},
   review={\MR{0429733}},
   doi={doi.org/10.1007/BFb0080437},
}

\end{biblist}
\end{bibdiv}

\end{document}